\documentclass[9pt,leqno,a4paper]{article}
\usepackage{xcolor}
\usepackage{amsfonts}
\usepackage{amsmath, amsfonts, amsthm, amssymb, amscd}
\usepackage[mathcal]{euscript}
\usepackage{mathrsfs}
\newtheorem{theorem}{Theorem}[section]

\newtheorem{lemma}[theorem]{Lemma}

\newtheorem{remark}[theorem]{Remark}
\numberwithin{equation}{section}

\usepackage{a4wide}

\newcommand \ih {\hat{i}}
\newcommand \ihr{{\color{red} {\ih}}}

\newcommand \ic {\check{i}}

\newcommand \jh {\hat{j}}
\newcommand \jhr {{\color{red} {\jh}}}

\newcommand \jc {\check{j}}
\newcommand \jcr {{\color{red} {\jc}}}

\newcommand \kh {\hat{k}}
\newcommand \khr {{\color{red} {\kh}}}

\newcommand \kc {\check{k}}
\newcommand \kcr {{\color{red} {\kc}}}

\newcommand \lh {\hat{l}}

\newcommand \lc {\check{l}}

\newcommand \alphar {{\color{red} {\alpha}}}
\newcommand \betar {{\color{red} {\beta}}}
\newcommand \gammar{{\color{red} \gamma}}
\newcommand \ar {{\color{red} {a}}}
\newcommand \br {{\color{red} {b}}}
\newcommand \cred {{\color{red} {c}}}

\newcommand \jr {{\color{red} {j}}}
\newcommand \kr {{\color{red} {k}}}

\newcommand \delb {\overline{\del}}
\newcommand \delu {\underline{\del}}
\newcommand \Au {\underline{A}}
\newcommand \Bu {\underline{B}}

\newcommand \Pu {\underline{P}}

\newcommand \Tu {\underline{T}}

\newcommand \Gt {\widetilde G}

\newcommand \del \partial
\newcommand \la \langle
\newcommand \ra \rangle

\let\oldmarginpar\marginpar
\renewcommand\marginpar[1]{\-\oldmarginpar[\raggedleft\footnotesize #1]%
{\raggedright\footnotesize #1}}

\begin{document}
\title{Global existence of small amplitude solution to nonlinear system of wave and Klein-Gordon equations in four space-time dimensions}
\author{Yue MA\footnote{Laboratoir Jacques-Louis Lions. Email: ma@ann.jussieu.fr}}
\maketitle
\begin{abstract}
In this article one will develop a so-called hyperboloidal foliation
method, which is an energy method based on a foliation of space-time
into hyperboloidal hypersurfaces. This method permits to treat the
wave equations and the Klein-Gordon equations in the same framework
so that one can apply it to the coupled systems of wave and
Klein-Gordon equations. As an application, one will establish the
global-in-time existence of small amplitude solution to the coupled
wave and Klei-Gordon equations with quadratic nonlinearity in four
space-time dimensions under certain conditions. Compared with those
introduced by S. Katayama, the conditions imposed in this article
permit to include some important nonlinear terms. All of these
suggests that this method may be a more natural way of regarding the
wave operator.
\end{abstract}
\tableofcontents
\section{Introduction}
One will consider Cauchy problems associated to a class of coupled
nonlinear wave and Klein-Gordon equations. The following is a
prototype:
\begin{equation}\label{intro Katayama}
\left\{
\aligned
&\Box u = N(\del u,\del u) + Q_1(\del u,\del v) + Q_2(\del v,\del v),
\\
&\Box v + v = Q_3(\del u,\del u) + Q_4(\del u, \del v) + Q_5(\del v,\del v),
\\
&u(B+1,x) = \varepsilon u_0,\quad v(B+1,x) = \varepsilon v_0,
\\
&\del_tu(B+1,x) = \varepsilon u_1,\quad \del_t v(B+1,x) = \varepsilon v_1.
\endaligned
\right.
\end{equation}
Here $u_i$ and $v_i$ are regular functions supported on disc $|x|\leq B$. $N(\cdot,\,\cdot)$ is a standard null quadratic form and $Q_i(\cdot,\,\cdot)$ are
arbitrary quadratic forms.

The method introduced in this article is a type of commuting vector field approach. With this techniques s. Klainerman has firstly established the global-in-time
existence of regular solution to nonlinear wave equations with null condition. (see \cite{Kl1} details). The idea of this
method also works when dealing with the nonlinear Klein-Gordon equations with quadratic nonlinearity (see \cite{Kl2} for
details). But when one attempts to try this idea on coupled system of wave and Klein-Gordon equations, one will face
some difficulties.
The main difficulty is that, one of the conformal Killing vector field of the wave equation, the scaling vector field $S:= t\del_t + r\del_r$ is {\bf not} a
conformal Killing vector field of Klein-Gordon equation, so that $S$ can not be used any longer. One may call this ``the difficulty of $S$".

In \cite{Ka}, S. Katayama has established the global-in-time existence of regular solution to \eqref{intro Katayama}.
To overcome ``the difficulty of $S$", Katayama has used an other version of Sobolev type estimate for replacing the classical Klainerman-Sobolev inequality, and
an technical $L^{\infty}-L^{\infty}$ estimates.

The main result of this article (see section \ref{main sec}) will improve part of the result of \cite{Ka}. The method
is not a generalization of the techniques introduced in \cite{Ka} but a new type of energy method based on a foliation of
space-time into hyperboloidal hypersurfaces. As far as the author is concerned, this foliation first appears in \cite{Ho},
where L. H\"ormander has developed an ``alternative energy method" for dealing the global existence of quasilinear Klein-Gordon equation. His observation is as
follows. Consider the following Cauchy problem associated to the linear Klein-Gordon equation in $\mathbb{R}^{n+1}$
\begin{equation}\label{intro eq linear}
\left\{
\aligned
&\Box u + a^2u = f,
\\
&u(B+1,x) = u_0,\quad u_t(B+1,x) = u_1,
\endaligned
\right.
\end{equation}
where $u_0,\, u_1$ are regular functions supported on $\{(B+1,x):|x|\leq B\}$ and $f$ is also a regular
function supported on
$$
\Lambda':= \{(t,x): |x|\leq t-1\},
$$
with $a,B>0$ two fixed positive constants.
By the Huygens' principle, the regular solution of \eqref{intro eq linear} is supported in
$$
\Lambda' \cap \{t\geq B+1\}.
$$
One denotes by:
$$
H_T := \{(t,x) : t^2 - x^2 = T^2,\, t>0\}
$$
and
$$
G_{B+1} = \Lambda'\cap \{(t,x): \sqrt{t^2- x^2}\geq B+1\},
$$
one can develop a hyperboloidal foliation of $G_{B+1}$, which is
$$
G_{2B} = H_T \times [B+1, \, \infty).
$$
Then, taking $\del_t u$ as multiplier, the standard procedure of energy estimate leads one to the following
energy inequality
$$
E_m(T,\,u)^{1/2} \leq E_m(B+1,\,u)^{1/2} + \int_{B+1}^T ds \bigg(\int_{H_T}f^2\bigg)^{1/2},
$$
where
$$
E_m(T,\,u) := \int_{H_T} \sum_{i=1}^{3}\big((x^i/t)\del_t u + \del_i u\big)^2 + ((T/t)\del_t u)^2 + (a/2)u^2\, dx .
$$
Then, H\"ormander has developed a Sobolev type estimate, see the lemma 7.6.1 of \cite{Ho}. Combined with the
energy estimate, he has managed to establish the decay estimate:
$$
\sup_{H_T} t^{n/2}|u| \leq \sum_{|I|\leq m_0} E_m(H_T,\,Z^I u)^{1/2}\leq E_m(H_{B+1},\,Z^I u)^{1/2} + \int_{B+1}^T ds \bigg(\int_{H_s}Z^I f^2\bigg)^{1/2},
$$
where $m_0$ is the smallest integer bigger the $n/2$.

But in the proof of \cite{Ho}, it seems that the only used term of the energy $E_m(H_T,\,u)$ is the last term $u^2$. The first two terms seem to be
omitted, at least when doing decay estimates. The new observation in this article is that the first two terms of the energy can also be used
for estimating some important derivatives of the solution. This leads one to the possibility of applying this
method on the case where $a=0$, which is the wave equation, so that the wave equations and the Klein-Gordon equations can
be treated in the same framework. This is the key of dealing the coupled wave and Klein-Gordon equations, and one may call it
hyperboloidal foliation method.

Here is the structure of this article. In section \ref{pre sec}, one will introduce the basic
theory of hyperboloidal foliation method including the energy estimates, the estimates on commutators and the decay estimates. For the convenience of proof, a new
frame, the so-called ``one frame" will be introduced to replace the classical ``null frame". The main result will be stated in section \ref{main sec}.
In this article
one will not cite any technical result. All tools used will be established in section \ref{pre sec}.
\section{Preliminaries}\label{pre sec}
\subsection{Notation}
First, one makes the following important conventions of index. The Latin index $a,b,c$ denote one of the positive integers $1,2,3$. The Greek index
$\alpha,\beta,\gamma$ denote one of the integers $0,1,2,3$. The Einstein's summation will be used. But to avoid possible confusion,
the dummy index will be printed in red.

One denotes by $H_T$ the hyperboloid $\{(t,x): t^2 - |x|^2 =T^2, t>0\}$ with its hyperbolic radius equals to $T$. $\Lambda'$ denotes the cone
$\{(t,x):|x|\leq t-1\}$. One denotes by $\mathcal {G}_{B+1}^{T}$ the region $\{(t,x): B+1\leq t^2-x^2\leq T^2\}$. Notice that in the region
$\Lambda' \cap \{|x|\leq t/2\}$, $t\leq \frac{2 \sqrt{3}}{3}T$.

One introduces the following vector fields:
$$
\aligned
&H_a := x^a\del_t + t\del_a,
\\
&\delb_a := t^{-1}H_a = \big(x^a/t\big)\del_t + \del_a.
\endaligned
$$

One sets $\mathscr{Z}$ a family of vector fields consists of $Z_{\alpha}$, where
$$
Z_{\alpha} := \del_{\alpha},
$$
and
$$
Z_{3+a} := H_a.
$$
One notices that for any $Z,\,Z'\in \mathscr{Z}$,
\begin{equation}\label{pre commutator Z}
[Z,Z'] \in \mathscr{Z},
\end{equation}
which means $\mathscr{Z}$ forms a Lie algebra.
For a general multi-index $I$, denote by $Z^I$ a $|I|-th$ order derivatives
$$
Z^I := Z_{I_1} \cdots Z_{I_{|I|}}.
$$
\subsection{Energy estimates}
One considers the following differential system:
\begin{equation}\label{pre eq energy}
\left\{
\aligned
&\Box w_i + G_i^{\jr\alphar\betar}\del_{\alphar\betar} w_{\jr} + {D_i}^2w_i = F_i,
\\
&w_i|_{H_{B+1}} = {w_i}_0,\quad \del_t w_i|_{H_{B+1}} = {w_i}_1.
\endaligned
\right.
\end{equation}
Here $G_i^{j\alpha\beta}$ and $F_i$ are regular functions supported in $\Lambda'$. ${w_i}_0,{w_i}_1$ are regular functions supported on $H_{B+1}\cap \Lambda'$.
To guarantee the hyperbolicity, one supposes that
\begin{equation}\label{pre condition symmetry}
G_i^{j\alpha\beta} = G_j^{i\alpha\beta},\quad G_i^{j\alpha\beta} = G_i^{j\beta\alpha}.
\end{equation}
$D_i$ are constants. $D_i = 0$ with $1\leq i\leq j_0$ and $D_i > \sigma >0$ when $j_0+1 \leq i\leq j_0 + k_0 =: n_0$.

One introduces the following ``standard" energy on hyperboloid $H_T$:
\begin{equation}\label{pre exression of energy}
\aligned
E_m(T,\,w_i)
&:= \int_{H_T}\bigg(|\del_t w_i|^2 + \sum_a|\del_a w_i|^2 + (2x^{\ar}/t)\del_tw_i \del_{\ar} w_i + 2(D_iu)^2 \bigg) dx,
\\
&= \int_{H_T} 2(D_iw_i)^2 + \sum_a |\delb_a w_i|^2 + \big((T/t)\del_t w_i\big)^2 dx,
\\
&= \int_{H_T} 2(D_iw_i)^2 + \sum_a \big((T/t) \del_a w_i\big)^2 + \sum_a\big((r/t)\del_a w_i + \omega^a \del_t w_i \big)^2 dx.
\endaligned
\end{equation}
And the ``curved" energy associated to the principal part of \eqref{pre eq energy}:
\begin{equation}\label{pre expression of curved energy}
E_G(s,w_i) :=
E_m(s,w_i) + 2\int_{H_s} \big(\del_t w_i \del_{\betar}w_{\jr} G_i^{\jr\alpha\betar}\big) \cdot (1,-x^a/t) dx
- \int_{H_s} \big(\del_{\alphar}w_i\del_{\betar}w_{\jr} G_i^{\jr\alphar\betar}\big)dx,
\end{equation}
here the second term on the right-hand-side is the Euclidian inner product of the vector \\
$\big(\del_t w_i \del_{\betar}w_{\jr} G_i^{\jr\alpha\betar}\big)_{\alpha}$ and the
vector $(1,-x^a/t)$.
\begin{remark}
From the structure of the strand energy on hyperboloid, one sees that the $L^2$ norm of $\delb_a w$ and $(T/s)\del_{\alpha}w$ are controlled directly.
These are (relatively) good derivatives. In general they enjoy better decay than $\del_{\alpha}w$. Notice that
$$
\sum_a \omega^a\big((r/t)\del_a + \omega^a\del_t\big)w = t^{-1}Sw.
$$
So the derivative $t^{-1}S$ is also good. But in this article it will not be used.
\end{remark}

Then in general the following result holds:
\begin{lemma}[Energy estimates]\label{pre lem energy}
Let $\{w_i\}$ be regular solution of \eqref{pre eq energy}. If the following estimates hold:
\begin{equation}\label{pre lem energy curved energy is big}
\sum_i E_m(s,w_i)\leq 3\sum_i E_G(s,w_i),
\end{equation}
\begin{equation}\label{pre lem energy curveterm is small}
\bigg|\int_{H_s}\frac{s}{t}\bigg( \del_{\alphar}G_i^{\jr\alphar\betar}\del_tw_i \del_{\betar}w_{\jr}
- \frac{1}{2}\del_tG_i^{\jr\alphar\betar}\del_{\alphar}w_i\del_{\betar}w_{\jr}\bigg) dx\bigg|
\leq M(s) E_m(s,w_i)^{1/2},
\end{equation}
\begin{equation}\label{pre lem energy source}
\bigg(\int_{H_s}\big|F_i\big|^2dx\bigg)^{1/2}ds \leq L_i(s).
\end{equation}
Then the following energy estimate hold:
$$
\bigg(\sum_i E_m(s,w_i)\bigg)^{1/2} \leq \bigg(\sum_i E_m(B+1,w_i)\bigg)^{1/2} + \sqrt{3}\int_{B+1}^s \sum_i L_i(\tau) + \sqrt{n_0}M(\tau)d\tau.
$$
\end{lemma}
\begin{remark}
In general the initial data is imposed on a plan rather than on a hyperboloid. But in Appendix A one will see that $E_m(B+1,w_i)$ is controlled by the $H^1(\mathbb{R}^3)$
norm of the initial data given on $\{B+1\}\times \mathbb{R}^3$.
\end{remark}
\begin{proof}
Under the assumptions \eqref{pre condition symmetry},
taking $\del_t w_i$ as multiplier, the standard energy estimate procedure gives
$$
\aligned
&\sum_i\bigg(\frac{1}{2} \del_t \sum_{\alpha}\big(\del_{\alpha} w_i\big)^2 + \sum_a\del_a\big(\del_a w_i \del_t w_i\big)
 + \del_{\alphar}\big(G_i^{\jr\alphar\betar}\del_t w_i\del_{\betar}w_{\jr}\big)
- \frac{1}{2}\del_t\big(G_i^{\jr\alphar\betar}\del_{\alphar}w_i\del_{\betar}w_{\jr}\big)\bigg)
\\
& = \sum_i \del_t w_i F_i
+ \sum_i\bigg(\del_{\alphar}G_i^{\jr\alphar\betar} \del_tw_i \del_{\betar}w_{\jr}
- \frac{1}{2}\del_tG_i^{\jr\alphar\betar}\del_{\alphar}w_i\del_{\betar}w_{\jr}\bigg).
\endaligned
$$
Then integrate in the region $\mathcal {G}_{B+1}^s$ and use the Stokes formulae,
$$
\aligned
&\frac{1}{2}\sum_i \big(E_G(s,w_i) -E_G(B+1,w_i)\big)
\\
&= \int_{\mathcal {G}_{B+1}^s} \del_t w_i F_i dx
+ \sum_i\int_{\mathcal {G}_{B+1}^s}\del_{\alphar}G_i^{\jr\alphar\betar} \del_tw_i \del_{\betar}w_{\jr} -
\frac{1}{2}\del_tG_i^{\jr\alphar\betar}\del_{\alphar}w_i\del_{\betar}w_{\jr}\,dx,
\\
&= \int_{B+1}^s (\tau/t)d\tau \int_{H_\tau} \del_t w_i F_i dx
+ \sum_i\int_{B+1}^s (\tau/t)d\tau \int_{H_\tau}\del_{\alphar}G_i^{\jr\alphar\betar} \del_tw_i \del_{\betar}w_{\jr} -
\frac{1}{2}\del_tG_i^{\jr\alphar\betar}\del_{\alphar}w_i\del_{\betar}w_{\jr}\,dx,
\endaligned
$$
which leads to
$$
\aligned
\quad \frac{d}{ds}\sum_i E_G(s,w_i)
&= 2\sum_i\int_{H_s}(s/t)\del_{\alphar}G_i^{\jr\alphar\betar} \del_tw_i \del_{\betar}w_{\jr}
-(s/2t)\del_tG_i^{\jr\alphar\betar}\del_{\alphar}w_i\del_{\betar}w_{\jr}\,dx
\\
&\quad +2\int_{H_s}(s/t) \del_t w_i F_i dx.
\endaligned
$$
So one gets
$$
\aligned
&\quad\bigg(\sum_i E_G(s,w_i)\bigg)^{1/2}\frac{d}{ds}\bigg(\sum_i E_G(s,w_i)\bigg)^{1/2}
\\
&\leq \sum_i\bigg(\int_{H_s}\big|F_i\big|^2dx\bigg)^{1/2} E_m(s,w_i)^{1/2} + M(s)\sum_i E_m(s,w_i)^{1/2}
\\
&\leq \sqrt{3}\bigg(\sum_i\int_{H_s}\big|F_i\big|^2dx\bigg)^{1/2} \bigg(\sum_iE_G(s,w_i)\bigg)^{1/2} + \sqrt{3}M(s)\sum_i E_G(s,w_i)^{1/2}
\\
&\leq \sqrt{3}\sum_i L_i(s) \bigg(\sum_iE_G(s,w_i)\bigg)^{1/2}
+ \sqrt{3n_0}M(s)\bigg(\sum_i E_G(s,w_i)\bigg)^{1/2}
\endaligned
$$
which leads to
$$
\quad\frac{d}{ds}\bigg(\sum_i E_G(s,w_i)\bigg)^{1/2}  \leq \sqrt{3}\sum_iL_i(s) + \sqrt{3n_0}M(s)
$$
By integrating on the interval $[B+1,s]$, the lemma is proved.
\end{proof}
\subsection{Commutators}
In this subsection one will establish the very important results of commutators.
Firstly, because $Z_{\alpha}$ are Killing vector fields of $\Box$, the following commutative relations hold:
$$
[\del_{\alpha}, \,\Box]=0,\quad [H_a,\, \Box] =0.
$$

The commutative relations between $H_a$ and $\delb_b$ are
\begin{equation}\label{pre commutator H-bar}
H_a \delb_b = \delb_b H_a - \frac{x^b}{t}\delb_a.
\end{equation}
The commutative relations between $\del_{\beta}$ and $H_a$ are:
\begin{equation}\label{pre commutator H-partial}
\aligned
&H_a\del_b = \del_b H_a - \delta_b^a\del_t,
\\
&H_a\del_t = \del_tH_a - \del_a.
\endaligned
\end{equation}
The commutative relations between $H_j$ and $(T/t)\del_{\alpha}$ are
\begin{equation}\label{pre commutator H-T/t}
\aligned
&H_a\bigg(\frac{T}{t}\del_t u\bigg) = -\frac{T}{t}\bigg(\del_a u + \frac{x^a}{t}\del_t u\bigg) + \frac{T}{t}\del_t(H_a u),
\\
&H_a\bigg(\frac{T}{t}\del_b u\bigg) = -\frac{T}{t}\bigg(\delta_b^a\del_t u + \frac{x^a}{t}\del_b u\bigg) + \frac{T}{t}\del_t(H_a u).
\endaligned
\end{equation}
The commutative relations between $\del_{\alpha}$ and $\delb_a$ are:
\begin{equation}\label{pre commutator partial bar}
\aligned
\del_b\delb_a = \delb_a \del_b + \delta_b^a t^{-1}\del_t,
\\
\del_t\delb_a = \delb_a \del_t - t^{-1}\frac{x^a}{t}\del_t.
\endaligned
\end{equation}
The commutative relations between $\del_{\alpha}$ and $(T/t)\delb_a$ are:
\begin{equation}\label{pre commutator partial T/t}
\aligned
\del_t \big((T/t)\del_{\alpha}\big) = (T/t)\del_{\alpha}\del_t - t^{-1}(T/t)\del_{\alpha},
\\
\del_a \big((T/t)\del_{\alpha}\big) = (T/t)\del_{\alpha}\del_a - \big(x^a/T^2\big)(T/t)\del_{\alpha}.
\endaligned
\end{equation}
One also needs the following commutative relations between $H_a$ and $\del_{\alpha}\del_{\beta}$:
\begin{equation}
\aligned
H_a\del_b\del_c &= \del_b\del_c H_a - \delta_a^b\del_t\del_c - \delta_a^c\del_t\del_b,
\\
H_a\del_t\del_b &= \del_t\del_b H_a - \del_t\del_b - \delta_a^b\del_{tt},
\\
H_a\del_{tt} &= \del_{tt} H_a - 2\del_{tt}.
\endaligned
\end{equation}
In general, one has the following estimates:
\begin{lemma}\label{pre lem commutator}
For any regular function $u$ supported in $\Lambda'$, the following estimates hold
\begin{equation}\label{pre lem commutator T/t}
\aligned
\big|Z^I\big((T/t)\del_{\alpha}u\big)\big| &\leq \big|(T/t)\del_{\alpha}Z^Iu\big| + C(n,|I|)\sum_{\beta,|J|<|I|}\big|(T/t)\del_{\beta}Z^Ju\big|,
\\
\big|(T/t)Z^I\del_{\alpha} u \big| &\leq \big|(T/t)\del_{\alpha}Z^Iu\big| + C(n,|I|)\sum_{\beta,|J|<|I|}\big|(T/t)\del_{\beta}Z^Ju\big|,
\endaligned
\end{equation}
\begin{equation}\label{pre lem commutator bar}
\big|Z^I\delb_a u\big| \leq \big|\delb_aZ^I u\big| + C(n,|I|)\sum_{b,|J|<|I|}\big|\delb_b Z^J u\big| + C(n,|I|)\sum_{\beta,|J|<|I|}\big|(T/t)\del_{\beta}Z^Ju\big|.
\end{equation}
\begin{equation}\label{pre lem commutator second-order}
\big|Z^I\del_{\alpha\beta} u \big| \leq \big|\del_{\alpha\beta} Z^Iu\big| + C(n,|I|)\sum_{\gamma,\gamma'\atop |J|<|I|} \big|\del_{\gamma\gamma'}Z^J u\big|.
\end{equation}
\begin{equation}\label{pre lem commutator second-order bar}
\big|[Z^I, \delu_a\del_{\beta}] u\big| + \big|[Z^I, \del_{\alpha}\delu_b] u\big|
 \leq C(n,I) \sum_{a,\beta\atop|J|\leq |I|-1}\big|\delu_a\del_{\beta}Z^Ju\big|
+ C(n,|I|)t^{-1}\sum_{\alpha,\beta\atop|J\leq|I|-1}|\del_{\alpha\beta}Z^J u|
\end{equation}
\end{lemma}
\begin{proof}
To prove \eqref{pre lem commutator T/t}, one needs the following identities:
\begin{equation}\label{pre commutator lem induction1}
\aligned
&H_b \big(x^a/t\big) = -\big(x^a/t\big)\big(x^b/t\big) + \delta_b^a,
\\
&\del_t\big(x^a/t\big) = -t^{-1}\big(x^a/t\big),
\\
&\del_b\big(x^a/t\big) = t^{-1}\delta_b^a,
\\
&H_b\big(t^{-1}\big) = -t^{-1}\big(x^b/t\big),
\\
&\del_t\big(t^{-1}\big) = -t^{-1}t^{-1}.
\endaligned
\end{equation}
Notice that in the cone $\Lambda'$, $x^a/t$ and $t^{-1}$ are bounded functions. Now one claims that:
\begin{equation}\label{pre commutator lem induction2}
\aligned
Z^I(T/t)\del_a u
&=(T/t)\del_aZ^I u
+ A(I,a)^{\alphar} \big(T/t\big)\del_{\alphar}Z^Ju + B(I,a)_{\ar}^{\alphar}\big(x^{\ar}/t\big)\del_{\alphar}Z^J u
\\
&\quad + \sum_{m\leq|J|} C(I,a)_m^ {\color{red}\alpha} t^{-m}\big(T/t\big)\del_{\alphar}Z^J u,
\endaligned
\end{equation}
where $A(I,a)^{\alpha}$, $B(I,a)_a^{\alpha}$ and $C(I,a)_m^{\alpha}$ are constants depending on $I,a,\alpha,m$. When $|I|=1$,
by \eqref{pre commutator H-T/t} and \eqref{pre commutator partial T/t},
$$
 Z_i(T/t)\del_a u
= A(i,a)^{\alphar}\big(T/t\big)\del_{\alphar}u + B(i,a)_{\ar}^{\alphar}\big(x^{\ar}/t\big)\del_{\alphar} u
+ C(i,a)_1^{\alphar}t^{-1}\big(T/t\big)\del_{\alphar} u,
$$
where $i = 0,1,\cdots 7$.
Suppose that \eqref{pre commutator lem induction2} holds for all multi-index $|I'|\leq N$,
Then for any multi-index $|I| = N+1$,
$$
\aligned
&Z^I(T/t)\del_a u
\\
&= Z^{I_1}Z^{I'}(T/t)\del_a u
\\
&=
Z^{I_1}\bigg((T/t)\del_aZ^I u
+ A(I,a)^{\alphar} \big(T/t\big)\del_{\alphar}Z^Ju + B(I,a)_{\ar}^{\alphar}\big(x^{\ar}/t\big)\del_{\alphar}Z^J u\bigg)
\\
& \quad+ Z^{I_1}\bigg(\sum_{m\leq|J|} C(I,a)_m^{\alphar} t^{-m}\big(T/t\big)\del_{\alphar}Z^J u\bigg),
\endaligned
$$
here $I_1$ represents the first component of $I$. Then by \eqref{pre commutator lem induction1},
One concludes that \eqref{pre commutator lem induction2} holds for any $|I|\leq N+1$. Then by induction on concludes that
the first estimate in \eqref{pre lem commutator T/t} holds.

The rest part of the lemma is proved in the same way. One omits the details.
\end{proof}
\subsection{Frames and Null conditions}
In this part one will introduce a so-called ``one frame", denoted by $\{\delu_{\alpha}\}$. This one frame will take the place
of the classical ``null frame". Define
$$
\delu_0 := \del_t, \quad \delu_a := \delb_a.
$$
The transition matrix between one frame and the natural frame is
$$
\delu_{\alpha} = \Phi_{\alpha}^{\betar}\del_{\betar}
$$
where
$$
\Phi:=
\begin{pmatrix}
&1        &0         &0         &0        \\
&x^1/t    &1         &0         &0        \\
&x^2/t    &0         &1         &0        \\
&x^3/t    &0         &0         &1        \\
\end{pmatrix}.
$$
Its inverse is
$$
\Psi:=
\begin{pmatrix}
&1        &0         &0         &0        \\
&-x^1/t   &1         &0         &0        \\
&-x^2/t   &0         &1         &0        \\
&-x^3/t   &0         &0         &1        \\
\end{pmatrix},
$$
so that
$$
\del_{\alpha} = \Psi_{\alpha}^{\betar}\delu_{\betar}.
$$
\begin{remark}
One notices that compared with the classical null frame, the norm of the only ``bad" direction $\delu_0$ is $1$,
while the only ``bad" direction $\underline{L}$ in the null frame is a null vector. That is the reason why one
calls it ``one-frame".
The advantage of this one frame, compared with the classical null frame is that the components of the transition
matrix are always regular in the cone $\Lambda'$.
\end{remark}
One may write a two tensor $\mathcal {T}$ under one frame or under the natural frame:
$$
\mathcal {T} = T^{\alphar\betar}\del_{\alphar}\del_{\betar} = \Tu^{\alphar\betar}\delu_{\alphar}\delu_{\betar},
$$
where $\Tu^{\alpha\beta}$ represent its components under one frame. In general one has the following estimates:
\begin{lemma}\label{pre lem frame}
For general two tensor $\mathcal {T}$, in the region $\Lambda'$
$$
\big|Z^I \Tu^{\alpha\beta}\big| \leq \sum_{\alpha',\beta'\atop |I'|\leq |I|} \big|Z^{I'}T^{\alpha'\beta'}\big|
$$
\end{lemma}
\begin{proof}
One notices that in $\Lambda'$, $\big|\del_{\alpha} \Phi_{\beta}^{\gamma}\big| \leq 1.$ The proof is just a calculation.
\end{proof}

Any two-order differential operator $T^{\alphar\betar}\del_{\alphar}\del_{\betar}$, can also be written under this one-frame.
\begin{equation}\label{pre frame change of frame}
T^{\alphar\betar}\del_{\alphar}\del_{\betar}u = \Tu^{\alphar\betar}\delu_{\alphar}\delu_{\betar} u
- \Tu^{\alphar\betar}\big(\delu_{\alphar}\Phi_{\betar}^{\betar'}\big)\del_{\betar'}u.
\end{equation}
Especially for the wave operator, one has the following expression:
$$
\Box u = \underline{m}^{\alphar\betar}\delu_{\alphar}\delu_{\betar} u
- \underline{m}^{\alphar\betar}\big(\delu_{\alphar}\Phi_{\betar}^{\betar'}\big)\del_{\betar'}u,
$$
Simple calculation gives $\underline{m}^{00} = T^2/t^2$, so one gets the following important identity:
\begin{equation}\label{pre expression of wave under oneframe}
(T/t)^2\delu_0\delu_0 u =
\Box u - \underline{m}^{0\ar}\delu_0 \delu_{\ar} u - \underline{m}^{\ar 0}\delu_{\ar} \delu_0 u -\underline{m}^{\ar\br}\delu_{\ar\br}u
+ \underline{m}^{\alphar\betar}\big(\delu_{\alphar}\Phi_{\betar}^{\betar'}\big)\del_{\betar'}u,
\end{equation}
here $\underline{m}^{0a} = x^a/t$, $\underline{m}^{ab} = \delta_a^b$ and
$$
\big|\delu_{\alpha}\Phi_{\beta}^{\beta'}\big| \leq Ct, \quad \text{in}\quad \Lambda'.
$$
So one concludes by the following lemma
\begin{lemma}\label{pre lem estimate of 00}
Let $u$ be a regular function supported in $\Lambda'$. Then
$$
(T/t)^2\big|\delu_0\delu_0 u\big|\leq \big|\Box u\big| + 2\sum_{a,\beta} \big|\delu_a \delu_{\beta} u\big| + Ct^{-1}\sum_{\beta}\big|\del_{\beta}u\big|.
$$
\end{lemma}
\begin{proof}
This is a direct result of \eqref{pre expression of wave under oneframe} and \eqref{pre commutator partial bar},
\end{proof}

Now a version of the classical null conditions will be introduced.
A quadratic form $T^{\alphar\betar}\xi_{\alphar}\xi_{\betar}$ is said to satisfy the null conditions if for any $\xi\in \mathbb{R}^4$ such that
$$
\xi_0\xi_0 - \sum_a \xi_a\xi_a =0,
$$
then
$$
T^{\alphar\betar} \xi_{\alphar}\xi_{\betar}= 0.
$$
Similarly, a cubic form $A^{\alphar\betar\gammar}\xi_{\alphar}\xi_{\betar}\xi_{\gammar}$ is said to verify the null conditions if
$$
A^{\alphar\betar\gammar} \xi_{\alphar}\xi_{\betar}\xi_{\gammar}= 0.
$$

\begin{lemma}\label{pre lem one frame}
Suppose that $T^{\alphar\betar}\xi_{\alphar}\xi_{\betar}$ is a quadratic form which satisfies the null conditions. If $\big|T^{\alpha\beta}\big| \leq K$,
then for any multi-index, the following estimate holds in the cone $\Lambda'$:
$$
\big|Z^I\Tu^{00}\big| \leq CK(T/t)^2.
$$
Similarly, if a cubic form $A^{\alphar\betar\gammar}\xi_{\alphar}\xi_{\betar}\xi_{\gammar}$ who verifies the null conditions and
$\big|A^{\alpha\beta\gamma}\big|\leq K$, then in the cone $\Lambda'$,
$$
\big|Z^I\Au^{000}\big| \leq CK(T/t)^2.
$$
\end{lemma}
\begin{proof}
One defines $\omega_a = \omega^a := x^a/|x|$ and $\omega_0 = -\omega^0 := -1.$\footnote{Arising the index by Minkowski metric.} Then
$$
\omega_0\omega_0 - \sum_a \omega_a \omega_a = 0.
$$
Let $\chi(\cdot)$ be a $C^{\infty}$ function defined on $(0,+\infty)$, $\chi(x) = 0$ when $x\leq 1/3$ and $\chi(x) = 1$ when $x\geq 1/2$.
Now consider the component $\Tu^{00}$,
$$
\aligned
\Tu^{00}
&= T^{\alphar\betar}\Psi_{\alphar}^0\Psi_{\betar}^0
\\
&= T^{\alphar\betar}\Psi_{\alphar}^0\Psi_{\betar}^0 - \chi(r/t)T^{\alphar\betar}\omega_{\alphar}\omega_{\betar}
\\
&= T^{\alphar\betar}\big(\Psi_{\alphar}^0\Psi_{\betar}^0 - \chi(r/t)\omega_{\alphar}\omega_{\betar}\big).
\endaligned
$$
Taking into account the fact that when $r\geq t/3$, $Z^J \omega^a\leq C$ and $H_a(r/t) = \omega^a(T/t)^2$. One has, when $r\geq t/2$,
$$
Z^I\Tu^{00} = Z^I\bigg(-\sum_a T^{a0}\omega^a\frac{r-t}{t} + \sum_{a,b}T^{ab}\omega^a\omega^b\frac{r^2-t^2}{t^2}\bigg)\leq CK\big(T^2/t^2\big).
$$
When $r\leq t/2$, by simple calculation, one has:
$$
\big|Z^I\Tu^{00}\big| \leq CK.
$$
But because in $\Lambda'\cap \{r\leq t/2\}$, $t^2\leq \frac{2 \sqrt{3}}{3} T^2$, one concludes that
$$
\big|Z^I\Tu^{00}\big| \leq CK(T/t)^2.
$$
The proof of the result on cubic forms is similar. One omits the details.
\end{proof}
\subsection{Decay estimates}

To turn the $L^2$ type energy estimates into the $L^{\infty}$ type estimates,
one needs the following Sobolev inequality, which is introduced as lemma 7.6.1 of
\cite{Ho}.
\begin{lemma}[Sobolev-type estimate on hyperboloid]
\label{pre lem sobolev}
Let $p(n)$ be the smallest integer $>n/2$. Any $C^{\infty}$ function defined on $\mathbb{R}^{1+n}$
satisfies
\begin{equation}
\label{pre ineq sobolev}
\sup_{H_T}t^n|u(t,x)|^2 \leq C(n) \sum_{I\leq |p(n)|}||Z^I u||^2_{L^2(H_T)}
\end{equation}
where $C(n)>0$ is a constant depending only on dimension $n$.
\end{lemma}

Combine this Lemma with the lemma of commutators \ref{pre lem commutator}, one gets the following results
\begin{lemma}\label{pre lem decay}
Let $u$ be a regular function supported in $\Lambda'$. Then the following estimates holds
$$
\aligned
&\sup_{H_T}\big|t^{(n-2)/2}T^{-1}\del_{\alpha}u\big|^2 \leq \sum_{|I|\leq p(n)}E_m(T,Z^I u),
\\
&\sup_{H_T}\big|t^{n/2}\delb_a  u\big|^2 \leq \sum_{|I|\leq p(n)}E_m(T,Z^I u),
\\
&\sup_{H_T}\big|D_i u\big|^2 \leq \sum_{|I|\leq p(n)}E_m(T,Z^I u).
\endaligned
$$
\end{lemma}
\begin{proof}
One recalls the equation \eqref{pre exression of energy}. Then it is a trivial result by lemma \ref{pre lem commutator}.
\end{proof}

Now let us consider the homogeneous linear wave equation
$$
\Box w = 0,\quad w|_{H_{B+1}} = w_0,\quad \del_t w|_{H_{B+1}} = w_1.
$$
where $w_i$ are regular functions supported on $H_{B+1}\cap \Lambda'$. By energy estimate lemma \ref{pre lem energy}, the associated energy $E_m(s,Z^Iw)$
is conserved.
Then by estimates of commutators \ref{pre lem commutator} and the sobolev lemma \ref{pre lem sobolev}, one gets
$$
\big|\delb_a w\big| \leq C(n)t^{-n/2},\quad \big|\del_{\alpha} w\big| \leq C(n)t^{-n/2+1}T^{-1}.
$$
This is exactly the classical result. But one notices that neither the explicit expression of the solution nor the scaling vector field $S=r\del_r + t\del_t$
is used.

Now one will turn to the energy and decay estimates of the some ``good" second-order derivatives, which are the derivatives such as $\delb_a \del_{\alpha} u.$ As we will see,
these derivatives have better decay than
that of $\del_{\alpha} u$ or even $\delb_a u$. In general one has
\begin{lemma}\label{pre lem decay high order}
Let $u$ be a regular function supported in the region $\Lambda'$. The following estimates hold:
$$
\sup_{H_T}\big|t^{n/2}T\delb_a\delb_{\alpha} u\big|^2 + \sup_{H_T}\big|t^{n/2}T\delb_{\alpha} \delb_a u\big|^2 \leq C(n)\sum_{|I|\leq p(n)+1} E_m(T,Z^Iu),
$$
$$
\int_{H_T}\big|T\delb_a\delb_{\alpha} u\big|^2 dx + \int_{H_T}\big|T\delb_{\alpha}\delb_a u\big|^2 dx \leq C\sum_{|I|\leq 1}E_m(T,Z^I u).
$$
\end{lemma}
\begin{proof}
Notice that
$$
\delu_a = \delb_a = t^{-1}H_a,
$$
one gets
$$
|\delu_a\del_{\alpha}u|\leq t^{-1}|H_a\del u|.
$$
Then by lemma \eqref{pre lem decay} one gets the first estimate. The second is a trivial result of the expression \eqref{pre exression of energy}.
\end{proof}
\begin{remark}
The energy estimates and decay estimates of $\delu_0\delu_0 u$ will consult the wave equation it-self. Roughly saying, by lemma \ref{pre lem estimate of 00}.
From here one can see that for wave equation, all the second-order derivatives do enjoy better decay compared with the gradient of the solution.
\end{remark}
At the end of this section, one gives the decay and energy estimates of the solution it-self.
\begin{lemma}\label{pre lem u it-self}
Let $u$ be a regular function supported in the cone $\Lambda'$. Then for any multi-index $|I|\geq 1$,
\begin{equation}\label{pre lem u it-self energy}
\int_{H_s}\big|t^{-1}Z^Iu\big|^2dx \leq C\sum_{|J|\leq |I|-1}E_m(s,Z^J u).
\end{equation}
For any multi-index $J$, if $\sum_{|I|\leq |J|+p(n)} E_m(s,Z^I u)^{1/2}\leq C'T^{\varepsilon}$ for an $\varepsilon \geq 0$, then
\begin{equation}\label{pre lem u it-self decay}
\big|Z^J u\big| \leq CC't^{-n/2}T^{1+\varepsilon}.
\end{equation}
\end{lemma}
\begin{proof}
\eqref{pre lem u it-self energy} is proved as follows.
When the operator $Z^I$ contains one factor $\del_{\alpha}$, by lemma \ref{pre exression of energy} and \eqref{pre commutator H-partial},
$$
\int_{H_s}\big|(s/t)Z^Iu\big|^2dx \leq C\sum_{|J|\leq |I|-1}E_m(s,Z^J u).
$$
When all of the factor of $Z^I$ are $H_a$, notice that $t^{-1}H_a = \delb_a$, by \eqref{pre exression of energy}
$$
\int_{H_s}\big|t^{-1}Z^Iu\big|^2dx = \int_{H_s}\big|\delb_aZ^{I'}u\big|^2dx \leq C\sum_{|J|\leq |I|-1}E_m(s,Z^J u).
$$
When $\sum_{|I|\leq |J|+p(n)} E_m(s,Z^I u)^{1/2}$ is bounded by $C'T^{\varepsilon}$, by lemma \ref{pre lem decay}, in the cone $\Lambda'$,
$$
|\del_r u| \leq C(n)t^{-(n-1-\varepsilon)/2}(t-r)^{-(1-\varepsilon)/2}.
$$
then the proof of \eqref{pre lem u it-self decay} is a integration along the radical direction.
\end{proof}
\section{Main result}\label{main sec}
One will consider the Cauchy problem associated to the following coupled wave and Klein-Gordon equations with quadratic nonlinearity:
\begin{equation}\label{main eq main}
\left\{
\aligned
&\Box w_i + G_i^{\jr\alphar\betar}(w,\del w)\del_{\alphar\betar}w_{\jr} + D_i^2 w_i = F_i(w,\del w),
\\
&w_i(B+1,x) = \varepsilon {w_i}_0,\quad w_i(B+1,x) = \varepsilon {w_i}_1.
\endaligned
\right.
\end{equation}
Here $1\leq i \leq n_0$. $D_i$ are constants. $D_i=0$ for $1\leq i \leq j_0$ and $D_i>0$ for $j_0+1\leq i \leq n_0$.

For the convenience of proof, one makes the following conventions of index. The Latin index $i,j,k,l$ denote one of the integer $1,2,3,\cdots,n_0$. the Latin
index with a circumflex accent on it such as $\ih,\jh,\kh,\lh$ denote one of the integer $1,2,3,\cdot, j_0$. The Latin index with a hacek on it such as
$\ic,\jc,\kc,\lc$  denote one of the integer $j_0+1, j_0+2,\cdots,n_0$.

$G_i^{j\alpha\beta}(\cdot,\cdot)$ and $F_i(\cdot,\cdot)$ are regular functions such that:
$$
\aligned
&G_i^{j\alpha\beta}(w,\del w) = A_i^{j\alpha\beta\gammar \kr}\del_{\gammar} w_{\kr} + B_i^{j\alpha\beta \kr}w_{\kr}.
\\
&F_i(w,\del w) = P_i^{\alphar\betar\jr\kr}\del_{\alphar}w_{\jr} \del_{\betar}w_{\kr} + Q_i^{\alphar\jr\kr} w_{\kr}\del_{\alphar}w_{\jr} + R_i^{\jr\kr}w_{\jr}w_{\kr}.
\endaligned
$$
Here $A_i^{j\alpha\beta\gamma k},B_i^{j\alpha\beta k},P_i^{\alpha\beta jk},Q_i^{\alpha jk},R_i^{jk}$ are constants with absolute value bounded by $K$.
One impose the following null conditions
\begin{equation}\label{main conditions null}
A_{\ih}^{\jh\alphar\betar\gammar \kh}\xi_{\alphar}\xi_{\betar}\xi_{\gammar}
= B_{\ih}^{\jh\alphar\betar \kh}\xi_{\alphar}\xi_{\betar}
= P_{\ih}^{\alphar\betar\jh\kh}\xi_{\alphar}\xi_{\betar} = 0, \quad \text{for any}\quad \xi_0\xi_0 - \sum_a\xi_a\xi_a = 0.
\end{equation}
One also supposes that
\begin{equation}\label{main conditions no-u}
B_i^{\jc\alpha\beta\kh}=Q_i^{\alpha j\kh} = R_i^{j\kh} = R_i^{\jh k} = 0.
\end{equation}
The initial data ${w_i}_0$ and ${w_i}_1$ are supposed to be ($C^{\infty}$) regular functions supported on the disc $\{|x|\leq B+1\}$.

Now one is ready to state the main theorem.
\begin{theorem}\label{main thm main}
Suppose \eqref{main conditions null} and \eqref{main conditions no-u} hold. Then there exists an $\varepsilon_0>0$ such that for any
$0\leq \varepsilon\leq \varepsilon_0$, \eqref{main eq main} has a unique global-in-time regular solution. In this case, the hyperbolic energy associated to the
wave components is conserved.
$$
\sum_{\ih,|I|\leq 3}E_m(s,Z^I u_{\ih})^{1/2} \leq C(\varepsilon_0).
$$
\end{theorem}

\begin{remark}

One improvement is that in \cite{Ka}, the system is not allowed to contain the term $\del u_{\ih} \del_t\del_t u_{\jh}$ and $u_{\ih}\del^2 u_{\jh}$
In this article this restriction is relaxed. One only demand classical null conditions on these terms.

The most important improvement is that in the proof one will used nothing technical but only the tools one has prepared in section \ref{pre sec}.
\end{remark}

\begin{remark}
This theorem also holds in the case where ${w_i}_0$ and ${w_i}_1$ are not $C^{\infty}$ (but still compactly supported). This is by a standard procedure of approximation.
In general one only demand that ${w_i}_0,\,{w_i}_1 \in H^4$. That is because following proof consults only the derivatives of solution of order $\leq 4$. This is
 also an improvement compared with \cite{Ka}, where the proof consults at least the $19^{th}$-order derivatives.
\end{remark}

\begin{remark}
When $n_0=j_0$, the theorem reduced to the classical result of global existence of regular solution to quasilinear wave equation with null conditions, see \cite{Ho}
for example. One can check that with out the Klein-Gordon components $v_i$, the following proof becomes very short and trivial which is simpler than the
classical one. Furthermore, the energy $E_m(s,Z^{I^*}u_{\ih})$ is conserved. This means the global solution is not only a ``small amplitude solution" but also
``small energy solution".
\end{remark}
\section{Proof of main result}\label{proof sec}
\subsection{Structure of the proof}
The proof is a standard boot-strap argument deviled into five parts. In the first part one supposes that in an interval $[0,T^*]$ the energy $E_m(s,Z^{I^*}w_i)$
is bounded for $0\leq s\leq T^*$. By lemma \ref{pre lem decay}, lemma \ref{pre lem decay high order} and lemma \ref{pre lem u it-self}, one gets the decay
estimates of $w_j,\,\delb w_j$ and $\delb\del w_j$. In part two with the those decay estimates, the $L^2$ norm of some source terms will be controlled. In part
three one gives the $L^2$ and decay estimates of the only ``bad" second-order derivative $\delu_0\delu_0 w$. In part four, equipped with the result of
part three one will give the $L^2$ estimates of the rest source terms. Then with all of these preparation, one will finally establish the main result in the last
part.
\subsection{Part one -- Energy assumption}
Suppose that on a interval $[B+1,T^*]$, the following energy assumptions hold with $0<\delta<1/6$:
\begin{equation}\label{proof energy assumption}
\aligned
&E_m(s,Z^{I^*} v_{\jc})^{1/2} \leq C_1\varepsilon s^{\delta}, \quad \text{for}\quad j_0+1\leq \jc\leq n_0,\quad 0\leq |I^*|\leq 4,
\\
&E_m(s,Z^{I^*} u_{\ih})^{1/2} \leq C_1\varepsilon s^{\delta},\quad \text{for}\quad 1\leq \ih\leq j_0, \quad|I^*|= 4,
\\
&E_m(s,Z^I u_{\ih})^{1/2} \leq C_1\varepsilon,\quad \text{for}\quad 1\leq \ih\leq j_0, \quad |I|\leq 3.
\endaligned
\end{equation}
By theorem \ref{appendix Thm A}, for any $C_1\varepsilon$, one can choose $\varepsilon'$ small enough such that
$$
\sum_{i\atop|I^*|\leq 4}E_m(B+1,Z^{I^*} w_i)^{1/2} < C_1\varepsilon,
$$
so that by continuity, $T^*>B+1$.

From \eqref{proof energy assumption} and lemma \ref{pre lem commutator}, the following $L^2$ estimates hold on $[B+1,T^*]$:
\begin{equation}\label{proof used energy estimates}
\aligned
&\sum_{\ih,\alpha\atop |I|\leq 3}\bigg(\int_{H_s}\big|(s/t)Z^I \del_{\alpha} u_{\ih}\big|^2 dx\bigg)^{1/2}
+ \sum_{\ih,a\atop |I|\leq 3}\bigg(\int_{H_s}\big| Z^I \delu_a u_{\ih}\big|^2 dx\bigg)^{1/2} \leq C C_1\varepsilon,
\\
&\sum_{\ih,\alpha\atop|I^*|\leq 4}\bigg(\int_{H_s}\big|(s/t)Z^{I^*} \del_{\alpha} w_i\big|^2 dx\bigg)^{1/2}
+ \sum_{\ih,a\atop |I^*|\leq 4}\bigg(\int_{H_s}\big| Z^{I^*} \delu_a w_{\ih}\big|^2 dx\bigg)^{1/2} \leq C C_1\varepsilon s^{\delta},
\\
&\sum_{\jc\atop|I^*|\leq 4}\bigg(\int_{H_s}\big|Z^{I^*}v_{\jc}\big|^2dx\bigg)^{1/2} \leq C C_1\varepsilon s^{\delta}.
\endaligned
\end{equation}
By lemma \ref{pre lem decay high order}, one also has the following $L^2$ estimates:
\begin{equation}\label{proof used energy estimates 2-order}
\aligned
&\sum_{a,\beta\atop \ih,|I|\leq 2}\bigg(\int_{H_s}\big|sZ^I \delu_a\delu_{\beta} u_{\ih}\big|^2 dx\bigg)^{1/2}
+\sum_{a,\beta\atop \ih,|I|\leq 2} \bigg(\int_{H_s}\big|sZ^I \delu_{\beta}\delu_a u_{\ih}\big|^2 dx\bigg)^{1/2} \leq C C_1\varepsilon,
\\
&\sum_{a,\beta\atop \ih,|I|\leq 3}\bigg(\int_{H_s}\big|sZ^I \delu_a\delu_{\beta} u_{\ih}\big|^2 dx\bigg)^{1/2}
+\sum_{a,\beta\atop \ih,|I|\leq 3} \bigg(\int_{H_s}\big|sZ^I \delu_{\beta}\delu_a u_{\ih}\big|^2 dx\bigg)^{1/2} \leq C C_1\varepsilon s^{\delta}.
\endaligned
\end{equation}
The following decay estimates come from lemma \ref{pre lem decay}. For $|J^*|\leq 2$ and $|J|\leq 1$,
\begin{equation}\label{proof decay preliminary}
\aligned
&\sup_{H_s}\Big(\big|s t^{1/2} \del_{\alpha}Z^{J^*} w_j\big| \Big) + \sup_{H_s}\Big(\big|t^{3/2}\big|\delu_a Z^{J^*} w_j\big| + t^{3/2}\big|Z^{J^*}v_{\kc}\big|\Big)
\leq C C_1\varepsilon s^{\delta},
\\
&\sup_{H_s}\Big(\big|s t^{1/2} \del_{\alpha}Z^J u_{\kh}\big| \Big) + \sup_{H_s}\Big(\big|t^{3/2}\big|\delu_a Z^J h_{\kh}\big|\Big)
\leq C C_1\varepsilon
\endaligned
\end{equation}
The following decay estimates will be more often used in the proof. The first inequality is due to \eqref{proof decay preliminary}
and lemma \ref{pre lem decay}. The second is due to \eqref{proof decay preliminary} and \ref{pre lem commutator T/t}. The last one  is due to
\eqref{pre lem commutator second-order bar}.
\begin{equation}\label{proof used decay}
\aligned
&\sup_{H_s}\big|s t^{1/2}Z^J \del_{\alpha} u_{\jh}\big| + \sup_{H_s}\big|t^{3/2}Z^J\delu_a  u_{\jh}\big| \leq C C_1\varepsilon,
\\
&\sup_{H_s}\Big(\big|s t^{1/2} Z^{J^*}\del_{\alpha} v_{\kc}\big| \Big) + \sup_{H_s}\Big(\big|t^{3/2} \big|Z^{J^*} \delu_av_{\kc}\big| + t^{3/2}\big|Z^{J^*}v_{\kc}\big|\Big)
\leq C C_1\varepsilon s^{\delta},
\\
&\sup_{H_s}\big|st^{3/2}\delu_a\delu_{\beta}Z^J u_{\jh}\big| + \sup_{H_s}\big|st^{3/2}\delu_{\alpha}\delu_bZ^J u_{\jh}\big|
\leq C C_1\varepsilon s^{\delta},
\\
&\sup_{H_s}\big|st^{3/2}\delu_a\delu_{\beta} u_{\jh}\big| + \sup_{H_s}\big|st^{3/2}\delu_{\alpha}\delu_b u_{\jh}\big|
\leq C C_1\varepsilon s,
\endaligned
\end{equation}
\subsection{Part two -- $L^2$ estimates}
In this step one will give $L^2$ type estimates of some quadratic terms which are components of the source terms.
In general one has the following $L^2$ estimates.
\begin{lemma}\label{proof lem L2 F 3}
Let $\{w_j\}$ be regular solution of \eqref{main eq main}.
Suppose that \eqref{proof used decay}, \eqref{proof used energy estimates} and \eqref{proof used energy estimates} hold. Let $\mathcal {A}_3$
be any of the following terms:
$$
v_{\kc}v_{\jc},\quad v_{\kc}\del_{\alpha}w_j,\quad \del_{\alpha}v_{\jc}\del_{\beta}w_k,
\quad\delu_a u_{\jh}\delu_{\beta} u_{\kh}.
$$
Then for any $|I|\leq 3$,
$$
\bigg(\int_{H_s}\big|Z^I\mathcal{A}_3\big|^2dx\bigg)^{1/2} \leq C(C_1\varepsilon)^2 s^{-3/2 + 2\delta}.
$$
Furthermore, if $\Gamma(t,x)$ is a regular function such that for any multi-index $J$, the following estimate holds in $\Lambda'$:
$$
\big|Z^J \Gamma\big| \leq C(J)(s/t),
$$
then for any $|I|\leq 4$,
$$
\bigg(\int_{H_s}\big|Z^I(\Gamma \delu_0 u_{\ih}\delu_0 u_{\jh})\big|^2dx\bigg)^2 \leq C(C_1\varepsilon)^2 s^{-3/2+2\delta}.
$$
Especially, for any $|I|\leq 3$,
$$
\bigg(\int_{H_s}\big|Z^IF_{\ih}\big|^2dx\bigg)^{1/2} \leq C(C_1\varepsilon)^2 s^{-3/2+2\delta},
$$
\end{lemma}
\begin{proof}
One begins with the estimates on $\mathcal {A}_3$. The proof is mainly a substitution of \eqref{proof used energy estimates},
\eqref{proof used energy estimates 2-order} and \eqref{proof used decay} into the corresponding expressions. Notice that when a product of derivatives $Z^I$ acts
on a product of two factor, there is always one factor derived less than $|I|/2$ times which may be controlled by \eqref{proof used decay}. Then the $L^2$ norm of
the hole product can be controlled by \eqref{proof used energy estimates} or \eqref{proof used energy estimates 2-order} when $|I|/2\geq p(3)$.
One writs the proof on $\del_{\alpha}v_{\jc}\del_{\beta}w_k$ and $\delu_a u_{\ih}\delu_{\beta}u_{\jh}$ in detail and omits the others.
Suppose that $|I|\leq 3$.
$$
\aligned
&\quad \bigg(\int_{H_s}\big|Z^I\big(\del_{\alpha}v_{\jc}\del_{\beta}w_k\big)\big|^2dx\bigg)^{1/2}
\\
&\leq  \sum_{|I_1|\leq 2\atop I_1 + I_2 = I}\bigg(\int_{H_s}\big|Z^{I_1}\del_{\alpha}v_{\jc}Z^{I_2}\del_{\beta}w_k\big|^2dx\bigg)^{1/2}
       +\bigg(\int_{H_s}\big|Z^I\del_{\alpha}v_{\jc}\del_{\beta}w_k\big|^2dx\bigg)^{1/2}
\\
&\leq \sum_{|I_1|\leq 2\atop I_1 + I_2 = I}
\bigg(\int_{H_s}\big|CC_1\varepsilon t^{-3/2}s^{\delta}(t/s)\big|^2\cdot\big|(s/t)Z^{I_2}\del_{\beta}w_k\big|^2dx\bigg)^{1/2}
\\
&\quad+\bigg(\int_{H_s}\big|Z^I\del_{\alpha}v_{\jc}\big|^2\cdot\big|CC_1\varepsilon t^{-1/2}s^{-1+\delta}\big|^2dx\bigg)^{1/2}
\\
&\leq \sum_{|I_2|\leq 3}
C(C_1\varepsilon)s^{-3/2+\delta}\bigg(\int_{H_s}\big|(s/t)Z^{I_2}\del_{\beta}w_k\big|^2dx\bigg)^{1/2}
+ C(C_1\varepsilon)s^{-3/2+\delta}\bigg(\int_{H_s}\big|Z^I\del_{\gamma}v_{\jc}\big|^2dx\bigg)^{1/2}
\\
&\leq C(C_1\varepsilon)^2 s^{-3/2+2\delta}.
\endaligned
$$

$$
\aligned
&\quad\bigg(\int_{H_s}\big|Z^I\big(\delu_a u_{\ih}\delu_{\beta}u_{\jh}\big)\big|^2dx\bigg)^{1/2}
\\
&\leq \sum_{|I_1|\leq 2\atop I_1 + I_2 = I}\bigg(\int_{H_s}\big|Z^{I_1}\delu_au_{\ih}Z^{I_2}\delu_{\beta}u_{\jh}\big|^2dx\bigg)^{1/2}
    + \bigg(\int_{H_s}\big|Z^I\delu_au_{\ih}\delu_{\beta}u_{\jh}\big|^2dx\bigg)^{1/2}
\\
&\leq\sum_{|I_1|\leq 2\atop I_1 + I_2 = I}\bigg(\int_{H_s}\big|CC_1\varepsilon t^{-3/2}s^{\delta}\big|^2\cdot(t/s)^2\big|(s/t)Z^{I_2}\delu_{\beta}u_{\jh}\big|^2dx\bigg)^{1/2}
\\
&\quad +CC_1\varepsilon s^{-3/2}\bigg(\int_{H_s}\big|Z^I\delu_au_{\ih}\big|^2dx\bigg)^{1/2}
\\
&\leq \sum_{|I_1|\leq 2\atop I_1 + I_2 = I}\bigg(\int_{H_s}\big|CC_1\varepsilon t^{-1/2}s^{-1+\delta}\big|^2\cdot\big|(s/t)Z^{I_2}\delu_{\beta}u_{\jh}\big|^2dx\bigg)^{1/2}
+C(C_1\varepsilon)^2 s^{-3/2}
\\
&\leq CC_1\varepsilon s^{-3/2+\delta}\sum_{|I_2|\leq |I|}\bigg(\int_{H_s}\big|(s/t)Z^{I_2}\delu_{\beta}u_{\jh}\big|^2dx\bigg)^{1/2}
+C(C_1\varepsilon)^2 s^{-3/2}
\\
&\leq C(C_1\varepsilon)^2 s^{-3/2+\delta}.
\endaligned
$$

Now one turns to the estimates of $\Gamma\delu_0 u_{\ih}\delu_0 u_{\jh}$. This quadratic forms is composed purely by the ``bad" derivative $\del_0$. But with the additional decay
provided by $\Gamma$, the $L^2$ estimates are still trivial:

$$
\aligned
&\quad \bigg(\int_{H_s}\big|Z^I\big(\Gamma\delu_0u_{\ih}\delu_0u_{\jh}\big)\big|^2dx\bigg)^{1/2}
\\
&\leq \sum_{|I_1|\leq 2\atop I_1 + I_2 + I_3 = I}\bigg(\int_{H_s}\big|Z^{I_3}\Gamma Z^{I_1}\delu_0u_{\ih}Z^{I_2}\delu_0u_{\jh}\big|^2dx\bigg)^{1/2}
\\
&\quad + \bigg(\int_{H_s}\big|\Gamma Z^I\delu_0u_{\ih}\delu_0u_{\jh}\big|^2dx\bigg)^{1/2}
\\
&\leq  C\sum_{|I_1|\leq 2\atop I_1 + I_2 + I_3 = I}\bigg(\int_{H_s}\big|CC_1\varepsilon t^{-1/2}s^{-1+\delta}\big|^2\big|(s/t)Z^{I_2}\delu_0u_{\jh}\big|^2dx\bigg)^{1/2}
\\
&\quad+C\bigg(\int_{H_s}\big|CC_1\varepsilon t^{-1/2}s^{-1}\big|^2\big|(s/t)Z^I\delu_0u_{\ih}\big|^2dx\bigg)^{1/2}
\\
&\leq C(C_1\varepsilon)^2 s^{-3/2+\delta}.
\endaligned
$$

The estimate on $Z^IF_{\ih}$ will consult the $L^2$ estimates proved. By definition,
$$
\aligned
F_{\ih} &= P_{\ih}^{\alphar\betar\jhr\khr}\del_{\alphar}u_{\jhr}\del_{\betar}u_{\khr}
\\
&+ P_{\ih}^{\alphar\betar\jcr\khr}\del_{\alphar}u_{\jcr}\del_{\betar}u_{\khr} + P_{\ih}^{\alphar\betar\jhr\kcr}\del_{\alphar}u_{\jhr}\del_{\betar}u_{\kcr}
 + P_{\ih}^{\alphar\betar\jcr\kcr}\del_{\alphar}u_{\jcr}\del_{\betar}u_{\kcr} + Q_{\ih}^{\alphar\jr\kcr}v_{\kcr}\del_{\alphar}u_{\jr}
 + R_{\ih}^{\jcr\kcr}v_{\jcr}v_{\kcr}.
\endaligned
$$
The first term can be written under one-frame:
$$
\aligned
P_{\ih}^{\alphar\betar\jhr\khr}\del_{\alphar}u_{\jhr}\del_{\betar}u_{\khr}
&= \Pu_{\ih}^{\alphar\betar\jhr\khr}\delu_{\alphar}u_{\jhr}\delu_{\betar}u_{\khr}
\\
&= \Pu_{\ih}^{00\jhr\khr}\delu_0u_{\jhr}\delu_0u_{\khr} + \Pu_{\ih}^{\ar 0\jhr\khr}\delu_{\ar}u_{\jhr}\delu_0u_{\khr}
  +\Pu_{\ih}^{0\br\jhr\khr}\delu_0u_{\jhr}\delu_{\br}u_{\khr} + \Pu_{\ih}^{\ar\br\jhr\khr}\delu_{\ar}u_{\jhr}\delu_{\br}u_{\khr}.
\endaligned
$$
By the null conditions \eqref{main conditions null} and lemma \ref{pre lem one frame},
$$
\big|Z^I\big(\Pu_{\ih}^{00\jh\kh}\big)\big|\leq C(I) (s/t)^2 \leq C(I)(s/t).
$$
So
$$
\bigg(\int_{H_s}\big|Z^I\big(\Pu_{\ih}^{00\jhr\khr}\delu_0u_{\jhr}\delu_0u_{\khr}\big)\big|^2dx\bigg)^{1/2} \leq C(C_1\varepsilon)^2 s^{-3/2+2\delta}.
$$
The rest terms of $F_{\ih}$ have been already estimated by the estimates on $\mathcal {A}_3$ terms. This completes the proof.
\end{proof}
Define
$$
\Gt_i^{\jhr\alphar\betar}\del_{\alphar\betar}u_{\jhr}
:= G_i^{\jhr\alphar\betar}\del_{\alphar\betar}u_{\jhr} - B_i^{\jhr\alphar\betar\khr}u_{\khr}\del_{\alphar\betar}u_{\jhr}
$$
and
$$
\Gt_i^{\jr\alphar\betar}\del_{\alphar\betar}w_{\jr} := \Gt_i^{\jhr\alphar\betar}\del_{\alphar\betar}u_{\jhr} + G_i^{\jcr\alphar\betar}\del_{\alphar\betar}v_{\jcr}.
$$
This is the ``good" part of $G$
\begin{lemma}\label{proof lem com 3}
Let $\{w_i\}$ be regular solution of \eqref{main eq main}. Suppose that \eqref{proof used decay}, \eqref{proof used energy estimates}
and \eqref{proof used energy estimates 2-order} hold. Then for any $|I|\leq 3$,
$$
\bigg(\int_{H_s}\big|[\Gt_{\ih}^{\jr\alphar\betar}\del_{\alphar}\del_{\betar},Z^I]w_j\big|^2dx\bigg)^{1/2}\leq C(C_1\varepsilon)^2s^{-3/2+2\delta}.
$$
\end{lemma}
\begin{proof}
Notice the following decomposition:
\begin{equation}\label{proof lem L2 com 3 eq1}
[\Gt_{\ih}^{\jr\alphar\betar}\del_{\alphar}\del_{\betar},Z^I]w_{\jr} = [\Gt_{\ih}^{\jhr\alphar\betar}\del_{\alphar}\del_{\betar},Z^I]u_{\jhr}
+[G_{\ih}^{\jcr\alphar\betar}\del_{\alphar}\del_{\betar},Z^I]v_{\jcr}.
\end{equation}
The second term is decomposed as following:
\begin{equation}\label{proof lem L2 com 3 eq2}
[G_{\ih}^{\jcr\alphar\betar}\del_{\alphar}\del_{\betar},Z^I]v_{\jcr}
= \sum_{|I_1|\geq 1\atop I_1+I_2=I}Z^{I_1}G_{\ih}^{\jcr\alphar\betar}Z^{I_2}\del_{\alphar}\del_{\betar}v_{\jcr}
+ G_{\ih}^{\jcr\alphar\betar}[\del_{\alphar}\del_{\betar},Z^I]v_{\jcr}
\end{equation}
Recall that
$$
|Z^JG_{\ih}^{j\alpha\beta}|\leq C(J)K\sum_{k,\gamma}\big|Z^J\del_{\gamma}w_k\big|.
$$
The $L^2$ norm of the first term in right-hand-side of \eqref{proof lem L2 com 3 eq2} can be estimated as follows:
\begin{equation}\label{proof lem L2 com 3 eq3}
\aligned
&\quad\sum_{|I_1|\geq 1\atop I_1+I_2=I}
\int_{H_s}\bigg(\big|Z^{I_1}G_{\ih}^{\jcr\alphar\betar}Z^{I_2}\del_{\alphar}\del_{\betar}v_{\jcr}\big|^2dx\bigg)^{1/2}
\\
&\leq \sum_{|I_1|= 1\atop I_1+I_2=I} + \sum_{|I_1|\geq 2\atop I_1+I_2=I}
\bigg(\int_{H_s}\big|Z^{I_1}G_{\ih}^{\jcr\alphar\betar}Z^{I_2}\del_{\alphar}\del_{\betar}v_{\jcr}\big|^2dx\bigg)^{1/2}
\\
&\leq \sum_{\alpha,\beta,\jc,|I_1|=1 \atop I_1+I_2=I}
\bigg(\int_{H_s}\big|CC_1\varepsilon t^{-1/2}s^{-1}\cdot Z^{I_2}\del_{\alpha}\del_{\beta}v_{\jc}\big|^2dx\bigg)^{1/2}
\\
&\quad+\sum_{\alpha,\beta,\jc,|I_2|\geq2 \atop I_1+I_2=I}
\bigg(\int_{H_s}\big|K(s/t)Z^{I_1}\del_{\alpha}w_j\cdot (t/s)CC_1\varepsilon t^{-3/2}s^{\delta}\big|^2dx\bigg)^{1/2}
\\
&\leq C(C_1\varepsilon)^2 s^{-3/2+2\delta}.
\endaligned
\end{equation}
The second term in right-hand-side of \eqref{proof lem L2 com 3 eq2} is estimated as follows. One notices that in the cone $\Lambda'$
$$
\big|[\del_{\alpha}\del_{\beta},Z^I]v_{\jc}\big|\leq C\sum_{\alpha,\beta,\atop |J|\leq |I|}\big|\del_{\alpha'}\del_{\beta'}Z^Jv_{\jc}\big|.
$$
So
$$
\aligned
\bigg(\int_{H_s}\big|G_{\ih}^{\jcr\alphar\betar}[\del_{\alphar}\del_{\betar},Z^I]v_{\jcr}\big|^2dx\bigg)^{1/2}
&\leq C\sum_{j,\alpha,\beta,\gamma\atop |J|\leq 4}
\bigg(\int_{H_s}\big|K\del_{\gamma}w_j\del_{\alpha'}\del_{\beta'}Z^Jv_{\jc}\big|^2dx\bigg)^{1/2}
\\
&\leq KC(C_1\varepsilon)^2s^{-3/2+2\delta}.
\endaligned
$$

The first term in right hand side of \eqref{proof lem L2 com 3 eq1} is decomposed as follows:
$$
[\Gt_i^{\jhr\alphar\betar}\del_{\alphar}\del_{\betar},Z^I]u_{\jhr} =
[A_i^{\jhr\alphar\betar\gammar\khr}\del_{\gammar}u_{\khr}\del_{\alphar}\del_{\betar},Z^I]u_{\jhr}
+[B_i^{\jhr\alphar\betar\kcr}v_{\kcr}\del_{\alphar}\del_{\betar},Z^I]u_{\jhr}
+[A_i^{\jhr\alphar\betar\gammar\kcr}\del_{\gammar}v_{\khr}\del_{\alphar}\del_{\betar},Z^I]u_{\jhr}
$$
The last two terms are finite linear combinations of $Z^{I_1}v_{\jc}Z^{I_2}\del_{\alpha}\del_{\beta}u_{\kh}$
and $Z^{I_1}\del_{\gamma}v_{\jc}Z^{I_2}\del_{\alpha}\del_{\beta}u_{\kh}$ with $I_1+I_2 = I$ and $|I_1|\geq 1$.
As in \eqref{proof lem L2 com 3 eq3}, their $L^2$ norms on $H_s$ can be estimated by $C(C_1\varepsilon)^2 s^{-3/2+2\delta}$. The first term can
be written under one-frame:
\begin{equation}\label{proof lem L2 com 3 eq4}
[A_i^{\jhr\alphar\betar\gammar\khr}\del_{\gammar}u_{\khr}\del_{\alphar}\del_{\betar},Z^I]u_{\jhr}
= [\Au_i^{\jhr\alphar\betar\gammar\khr}\delu_{\gammar}u_{\khr}\delu_{\alphar}\delu_{\betar},Z^I]u_{\jhr}
+ [\Au_i^{\jhr\alphar\betar\gammar\khr}\delu_{\gammar}u_{\khr}(\delu_{\alphar}\Phi_{\betar}^{\betar'})\del_{\betar'},Z^I]u_{\jhr}
\end{equation}
Recall that by null conditions \eqref{main conditions null} and lemma \ref{pre lem one frame}, one has
$$
\big|Z^I\Au_i^{\jh000\gamma\kh}\big|\leq C(I)(s/t)^2.
$$
The first term can be controlled as follows:
\begin{align*}
&\big|[\Au_i^{\jhr\alphar\betar\gammar\khr}\delu_{\gammar}u_{\khr}\delu_{\alphar}\delu_{\betar},Z^I]u_{\jhr}\big|
\\
&\leq \big|[\Au_i^{\jh000\khr}\delu_0u_{\khr}\delu_0\delu_0,Z^I]u_{\jhr}\big|
\\
&\quad + \big|[\Au_i^{\jhr\ar\betar\gammar\khr}\delu_{\gammar}u_{\khr}\delu_{\ar}\delu_{\betar},Z^I]u_{\jhr}\big|
+\big|[\Au_i^{\jhr\alphar\br\gammar\khr}\delu_{\gammar}u_{\khr}\delu_{\alphar}\delu_{\br},Z^I]u_{\jhr}\big|
+\big|[\Au_i^{\jhr\alphar\betar\cred\khr}\delu_{\cred}u_{\khr}\delu_{\ar}\delu_{\betar},Z^I]u_{\jhr}\big|.
\end{align*}
Then as in \eqref{proof lem L2 com 3 eq3}, its $L^2$ norm on $H_s$ is controlled
by $C(C_1\varepsilon)^2s^{-3/2+2\delta}$.

To estimate the $L^2$ norm on $H_s$ of the last term in right-hand-side of \eqref{proof lem L2 com 3 eq4}, one recalls that from definition,
$\big|Z^I\del_{\alpha}\Phi_{\beta}^{\beta'}\big|\leq Ct^{-1}$. Taking this into consideration and run the same method of \eqref{proof lem L2 com 3 eq3}
one sees that its $L^2$ norm on $H_s$ is also controlled by $C(C_1\varepsilon)^2s^{-5/2+\delta}$. Then finally the lemma is proved when taking the assumption
$\delta<1/6$ into account.
\end{proof}

\begin{lemma}\label{proof lem L2 F 4}
Let $\{w_j\}$ be regular solution of \eqref{main eq main}.
Suppose that \eqref{proof used decay}, \eqref{proof used energy estimates} and \eqref{proof used energy estimates} hold. Let $\mathcal {A}_4$ be any of the
following terms:
$$
v_{\kc}v_{\jc},\quad v_{\kc}\del_{\alpha}w_j,\quad \del_{\alpha}w_j\del_{\beta}w_k.
$$
Then for any $|I^*|\leq 4$,
$$
\bigg(\int_{H_s}\big|Z^{I^*}\mathcal{A}_4\big|^2dx\bigg)^{1/2} \leq C(C_1\varepsilon)^2 s^{-1 + \delta}.
$$

Especially, for any $|I^*|\leq 4$,
$$
\bigg(\int_{H_s}\big|Z^{I^*}F_i\big|\bigg)^{1/2}\leq C(C_1\varepsilon)^2 s^{-1+\delta}.
$$
\end{lemma}

\begin{proof}
The $L^2$ estimates of the terms $\mathcal {A}_4$ are nearly the same of those in the proof of lemma \ref{proof lem L2 F 3}. One will only prove the case where
$\mathcal {A}_4 = \del_{\alpha}w_j\del_{\beta}w_k$.
To estimate the term consulting $\del_{\alpha}v_{\jc}\del_{\beta}u_{\kh}$
$$
\aligned
&\bigg(\int_{H_s}\big|Z^{I^*}\big(\del_{\alpha}v_{\jc}\del_{\beta}u_{\kh}\big)\big|^2dx\bigg)^{1/2}
\\
&\leq \sum_{|I^*_1|\leq 2\atop I^*_1+I^*_2=I^*}\bigg(\int_{H_s}\big|Z^{I^*_1}\del_{\alpha}v_{\jc}Z^{I^*_2}\del_{\beta}u_{\kh}\big|^2dx\bigg)^{1/2}
     +\sum_{|I^*_2|\leq 1\atop I^*_1+I^*_2=I^*}\bigg(\int_{H_s}\big|Z^{I^*_1}\del_{\alpha}v_{\jc}Z^{I^*_2}\del_{\beta}u_{\kh}\big|^2dx\bigg)^{1/2}
\\
&\leq CC_1\varepsilon \sum_{|I^*_1|\leq 2\atop I^*_1+I^*_2=I^*}\bigg(\int_{H_s}\big|t^{-3/2}s^{\delta}(t/s)\big|^2\cdot\big|(s/t)Z^{I^*_2}\del_{\beta}u_{\kh}\big|^2dx\bigg)^{1/2}
\\
&\quad+CC_1\varepsilon\sum_{|I^*_2|\leq 1\atop I^*_1+I^*_2=I^*}\bigg(\int_{H_s}\big|(s/t)Z^{I^*_1}\del_{\alpha}v_{\jc}\big|^2(t/s)^2\cdot\big|t^{-1/2}s^{-1}\big|^2dx\bigg)^{1/2}
\\
&\leq C(C_1\varepsilon)^2 (s^{-3/2+2\delta} + s^{-1+\delta}).
\endaligned
$$

The terms $\del_{\alpha}u_{\jh}\del_{\beta}u_{\kh}$ are estimated as follows:

$$
\aligned
&\bigg(\int_{H_s}\big|Z^{I^*}\big(\del_{\alpha}u_{\jh}\del_{\beta}u_{\kh}\big)\big|^2dx\bigg)^{1/2}
\\
&\leq \sum_{|I^*_1|\leq 1\atop I^*_1+I^*_2 = I^*}\bigg(\int_{H_s}\big|Z^{I^*_1}\del_{\alpha}u_{\jh}Z^{I^*_2}\del_{\beta}u_{\kh}\big|^2dx\bigg)^{1/2}
+     \sum_{|I^*_2|\leq 1\atop I^*_1+I^*_2 = I^*}\bigg(\int_{H_s}\big|Z^{I^*_1}\del_{\alpha}u_{\jh}Z^{I^*_2}\del_{\beta}u_{\kh}\big|^2dx\bigg)^{1/2}
\\
&\quad + \sum_{|I^*_2|=2 \atop I^*_1+I^*_2 = I^*}\bigg(\int_{H_s}\big|Z^{I^*_1}\del_{\alpha}u_{\jh}Z^{I^*_2}\del_{\beta}u_{\kh}\big|^2dx\bigg)^{1/2}
\\
&\leq \sum_{|I^*_1|\leq 1\atop I^*_1+I^*_2 = I^*}\bigg(\int_{H_s}\big|CC_1\varepsilon t^{-1/2}s^{-1}\big|^2\cdot(t/s)^2\big|(s/t)Z^{I^*_2}\del_{\beta}u_{\kh}\big|^2dx\bigg)^{1/2}
\\
&\quad+\sum_{|I^*_2|\leq 1\atop I^*_1+I^*_2 = I^*}\bigg(\int_{H_s}\big|(s/t)Z^{I^*_1}\del_{\alpha}u_{\jh}\big|^2(t/s)^2\cdot\big|CC_1\varepsilon t^{-1/2}s^{-1}\big|^2 dx\bigg)^{1/2}
\\
&\quad+\sum_{|I^*_1|=2 \atop I^*_1+I^*_2 = I^*}\bigg(\int_{H_s}\big|(s/t)Z^{I^*_1}\del_{\alpha}u_{\jh}\big|^2(t/s)^2\cdot\big|CC_1\varepsilon t^{-1/2}s^{-1+\delta}\big|^2dx\bigg)^{1/2}
\\
&\leq C(C_1\varepsilon)^2 s^{-1+\delta}.
\endaligned
$$

The terms $\del_{\alpha}v_{\jc}\del_{\beta}v_{\kc}$ are estimated as follows:
$$
\aligned
&\bigg(\int_{H_s}\big|Z^{I^*}\big(\del_{\alpha}v_{\jc}\del_{\beta}v_{\kc}\big)\big|^2dx\bigg)^{1/2}
\\
&\leq \sum_{|I^*_1|\leq 2\atop I^*_1+I^*_2 = I^*}\bigg(\int_{H_s}\big|Z^{I^*_1}\del_{\alpha}v_{\jc}Z^{I^*_2}\del_{\beta}v_{\kc}\big|^2dx\bigg)^{1/2}
+     \sum_{|I^*_2|\leq 1\atop I^*_1+I^*_2 = I^*}\bigg(\int_{H_s}\big|Z^{I^*_1}\del_{\alpha}v_{\jc}Z^{I^*_2}\del_{\beta}v_{\kc}\big|^2dx\bigg)^{1/2}
\\
&\leq \sum_{|I^*_1|\leq 2\atop I^*_1+I^*_2 = I^*}\bigg(\int_{H_s}\big|CC_1\varepsilon t^{-3/2}s^{\delta}\big|^2\cdot(t/s)^2\big|(s/t)Z^{I^*_2}\del_{\beta}v_{\kh}\big|^2dx\bigg)^{1/2}
\\
&\quad+\sum_{|I^*_2|\leq 1\atop I^*_1+I^*_2 = I^*}\bigg(\int_{H_s}\big|(s/t)Z^{I^*_1}\del_{\alpha}u_{\jh}\big|^2(t/s)^2\cdot\big|CC_1\varepsilon t^{-3/2}s^{\delta}\big|^2 dx\bigg)^{1/2}
\\
&\leq C(C_1\varepsilon)^2 s^{-3/2+2\delta}\leq C(C_1\varepsilon)^2 s^{-1+\delta}.
\endaligned
$$

One notices that $F_i$ is a finite linear combination of $\del_{\alpha}w_j\del_{\beta}w_k,\,v_{\jc}\del_{\alpha}w_k$ and $v_{\jc}v_{\kc}$. By the estimates just proved,
the last estimate on $Z^IF_i$ is trivial.
\end{proof}
\begin{lemma}\label{proof lem L2 com 4}
Let $\{w_j\}$ be regular solution of \eqref{main eq main}. Suppose that \eqref{proof used energy estimates}, \eqref{proof used energy estimates 2-order} and
\eqref{proof used decay} hold. Then for any $|I^*|\leq 4$,
$$
\bigg(\int_{H_s}\big|[\Gt_i^{\jr\alphar\betar}\del_{\alphar}\del_{\betar},Z^{I^*}]w_{\jr}\big|^2dx\bigg)^{1/2}
\leq C(C_1\varepsilon)^2s^{-1+\delta}.
$$
\end{lemma}
\begin{proof}
The proof is exactly the same to that of lemma \ref{proof lem com 3}. The only thing one should pay attention to is that when $|I^*|=3$ the decay estimates and $L^2$ estimates on
$Z^{I^*}\del_{\alpha}\del_{\beta}v_{\jc}$ provided by
\eqref{proof used decay} and \eqref{proof used energy estimates} is not as good as in the case where $|I^*|\leq 2$ which is the case in the proof of lemma
\ref{proof lem com 3}. So here one has only a decay rate as $s^{-1+\delta}$.
\end{proof}

\subsection{Part three -- Energy and decay estimates of ``bad'' derivatives}
In this part one will give the energy and decay estimates of ``bad" second-order derivatives, which are the terms $\delu_0\delu_0Z^I u_{\jh}$.
The following result is an expression of $\delu_0\delu_0Z^I u_{\ih}$ given by other ``good'' terms, which is an algebraic transform of \eqref{main eq main}.
\begin{lemma}\label{proof lem 00 expression}
Let $\{w_i\}$ be solution of \eqref{main eq main}, then for any multi-index $I$ the following identity holds
\begin{equation}\label{proof 00 expression}
\aligned
&\quad (s/t)^2 \big(\delu_0\delu_0 Z^I u_{\ih} + (t/s)^2u_{\khr} \Bu_{\ih}^{\jhr00\khr} \delu_0\delu_0 Z^I u_{\jhr}\big)
\\
&= Z^IF_{\ih} - Z^I\big(\Gt_{\ih}^{\jr\alphar\betar}\del_{\alphar\betar}w_{\jr}\big) + [\Bu_{\ih}^{\jhr00\khr}u_{\khr}\delu_{00},Z^I]u_{\jhr}
\\
&\quad -
Z^I\big(\Bu_{\ih}^{\jhr\ar\br\khr}u_{\khr}\delu_{\ar\br} u_{\jhr} + \Bu_{\ih}^{\jhr\ar0\khr}u_{\khr}\delu_{\ar0} u_{\jhr} + \Bu_{\ih}^{\jhr0\br\khr}u_{\khr}\delu_{0\br} u_{\jhr}\big)
\\
&\quad -
\big(\underline{m}^{\ar\br}\delu_{\ar\br}Z^Iu_{\ih} + \underline{m}^{\ar0}\delu_{\ar0}Z^Iu_{\ih} + \underline{m}^{0\br}\delu_{0\br}Z^Iu_{\ih}\big)
\\
&\quad +
Z^I\big(\Bu_{\ih}^{\jhr\alphar\betar\khr}u_{\khr}(\delu_{\alphar}\Phi_{\betar}^{\betar'})\del_{\betar'}u_{\jhr}\big) + \underline{m}^{\alphar\betar}(\delu_{\alphar}\Phi_{\betar}^{\betar'})\del_{\betar'}Z^Iu_{\ih}
\\
& =:\mathcal{R}_{\ih}.
\endaligned
\end{equation}
Furthermore, there exists a universal constant $C^*$ such that when $|u_{\ih}|\leq K^{-1}C^*\leq 1$, the following estimate holds
\begin{equation}\label{proof basic estimate 00}
\big|(s/t)^2 \delu_{00}Z^Iu_{\ih}\big| \leq C\max_{\ih}|\mathcal {R}_{\ih}|,
\end{equation}
where $C$ is a universal constant
\end{lemma}
\begin{proof}
One can write \eqref{main eq main} under the following form:
$$
\Box u_{\ih} + B_{\ih}^{\jhr\alphar\betar\khr}u_{\khr}\del_{\alphar\betar}u_{\jhr} = F_{\ih} -\Gt_{\ih}^{\jr\alphar\betar}\del_{\alphar\betar}w_{\jr}.
$$
Then write the term $B_{\ih}^{\jhr\alphar\betar\khr}u_{\khr}\del_{\alphar\betar}u_{\jhr}$ under one frame, by \eqref{pre frame change of frame},
$$
\Box u_{\ih} + \Bu_{\ih}^{\jhr\alphar\betar\khr}u_{\khr}\del_{\alphar\betar}u_{\jhr}
= F_{\ih} -\Gt_{\ih}^{\jr\alphar\betar}\del_{\alphar\betar}w_{\jr} + B_{\ih}^{\jhr\alphar\betar\khr}u_{\khr}(\delu_{\alphar}\Phi_{\betar}^{\betar'})\del_{\betar'}u_{\jhr}.
$$
Then derive it with respect to an arbitrary product $Z^I$:
$$
\aligned
&\Box Z^I u_{\ih} + \Bu_{\ih}^{\jhr00\jhr}u_{\khr}\delu_{00}Z^Iu_{\jhr}
\\
& = -Z^I\big(\Bu_{\ih}^{\jhr\ar\br\jhr}u_{\khr}\delu_{\ar\br}u_{\jhr} + \Bu_{\ih}^{\jhr\ar 0\jhr}u_{\khr}\delu_{\ar 0}u_{\jhr} + \Bu_{\ih}^{\jhr0\ar\jhr}u_{\khr}\delu_{0\ar}u_{\jhr}\big)
\\
&\quad + Z^IF_{\ih} - Z^I\big(\Gt_{\ih}^{\jr\alphar\betar}\big(\del_{\alphar\betar} w_{\jr}\big)
+ Z^I\big(B_{\ih}^{\jhr\alphar\betar\khr}u_{\khr}(\delu_{\alphar}\Phi_{\betar}^{\betar'})\del_{\betar'}u_{\jhr}\big)
 + [\Bu_{\ih}^{\jhr00\khr}u_{\khr}\delu_{00},Z^I]u_{\jhr} .
\endaligned
$$
By \eqref{pre expression of wave under oneframe}, one gets \eqref{proof 00 expression}.

Consider the linear algebraic equations given by \eqref{proof 00 expression},
$$
(s/t)^2 \delu_{00}Z^Iu_{\ih} + \big((t/s)^2B_{\ih}^{\jhr00\khr}u_{\khr}\big)\big((s/t)^2\delu_{00}Z^Iu_{\jhr}\big) = \mathcal {R}_{\ih}.
$$
By lemma \ref{pre lem one frame} and \eqref{proof used decay},
$$
\big|(t/s)^2B_{\ih}^{\jhr00\khr}u_{\khr}\big| \leq CK\max_{\kh}|u_{\kh}|,
$$
where $C$ is a universal constant. When $CK\max_{\kh}|u_{\kh}|\leq 1/2$, that is $\max_{\kh}|u_{\ih}|\leq (2KC)^{-1}$,
by basic linear algebra, one has the following estimates:
$$
\big|(s/t)^2 \delu_{00}Z^Iu_{\ih}\big| \leq C\max_{\ih}|\mathcal {R}_{\ih}|,
$$
where $C^*$ is also a universal constant.
\end{proof}
\begin{remark}
In the expression of $\mathcal{R}$, the first term is a linear term while the rest are quadratic terms. The linear part are composed by the ``good'' second-order derivatives so that one can
deduce from here a better decay of $(s/t)^2\del_0\del_0 Z^I u_{\ih}$.
\end{remark}
Now one turns to the decay estimate of $\del_0\del_0 Z^Iu$.
\begin{lemma}\label{proof lem 00 decay}
Let $\{w_i\}$ a regular solution of \eqref{main eq main}. Suppose that \eqref{proof used decay} and \eqref{proof used energy estimates}
hold with $C_1\varepsilon \leq 1$. Then for any $|J|\leq 1$,
$$
\big|(s/t)^2  \delu_0\delu_0 Z^Ju_{\jh}\big|\leq C(K+1)C_1\varepsilon t^{-3/2}s^{-1+2\delta},
$$
and
$$
\big|(s/t)^2  Z^J\delu_0\delu_0 u_{\jh}\big|\leq C(K+1)C_1\varepsilon t^{-3/2}s^{-1+2\delta}.
$$
\end{lemma}
\begin{proof}
Take the notation of in lemma \ref{proof lem 00 expression}.
The proof is mainly a $L^{\infty}$ estimate of $\mathcal{R}_{\ih}$.

One notices that $Z^JF_{\ih}$ is a finite linear combination of $\del w_i\del w_j,\, v_{\kc}\del w_j$ and $v_{\kc}v_{\jc}$. By \eqref{proof used decay}, one easily
gets
$$
\big|Z^JF_{\ih}\big| \leq C(C_1\varepsilon)^2\big(t^{-1}s^{-2+2\delta} + t^{-2}s^{-1+2\delta} + t^{-3}s^{2\delta}\big)\leq CK(C_1\varepsilon)^2 t^{-1}s^{-2+2\delta}.
$$
Similarly, $\Gt_{\ih}^{\jr\alphar\betar}\del_{\alphar\betar}w_{\jr}$ is a finite linear combination of $v_{\kc}\del_{\alpha\beta}w_j$ and $\del_{\gamma}w_j\del_{\alpha\beta}w_k$.
By \eqref{proof used decay}
$$
\big|Z^J\big(\Gt_{\ih}^{\jr\alphar\betar}\del_{\alphar\betar}w_{\jr}\big)\big|
\leq CK(C_1\varepsilon)^2\big(t^{-2}s^{-1+2\delta} + t^{-1}s^{-2+2\delta}\big)
\leq CK(C_1\varepsilon)^2 t^{-1}s^{-2+2\delta}.
$$

By lemma \ref{pre lem frame} $\big|Z^J\Bu_{\ih}^{\jhr\alphar\betar\khr}\big| \leq CK$. Then by the last inequality of \eqref{proof used decay}, one has,
$$
\big|Z^I\big(\Bu_{\ih}^{\jhr\ar\betar\khr}u_{\khr}\delu_{\ar\betar}u_{\jhr}\big)\big| + \big|Z^I\big(\Bu_{\ih}^{\jhr\alphar\br\khr}u_{\khr}\delu_{\alphar\br}u_{\jhr}\big)\big|
\leq CK(C_1\varepsilon)^2 t^{-3}s^{\delta}.
$$
Similarly,
$$
\big|\underline{m}^{\ar\betar}\delu_{\ar\betar}Z^Iu_{\ih}\big| + \big|\underline{m}^{\alphar\br}\delu_{\alphar\br}Z^Iu_{\ih}\big| \leq CC_1\varepsilon t^{-3/2}s^{-1+\delta}.
$$

One also notices that $\big|Z^J\delu_{\alpha}\big(\Phi_{\beta}^{\beta'}\big)\big|\leq Ct^{-1}$. Then
$$
\big|Z^J\big(\Bu_{\ih}^{\jhr\alphar\betar\khr}u_{\khr}\big(\delu_{\alphar}\Phi_{\betar}^{\betar'}\big)\del_{\betar'}u_{\jhr} \big)\big|
\leq CKC_1\varepsilon t^{-3}s^{\delta}.
$$
Similarly,
$$
\big|\underline{m}^{\alphar\betar}\big(\delu_{\alphar}\Phi_{\betar}^{\betar'}\big)\del_{\betar'}Z^Ju_{\ih}\big| \leq CC_1\varepsilon t^{-3/2}s^{-1}.
$$

Now one will consider the term $[\Bu_{\ih}^{\jhr00\khr}u_{\khr},Z^J]u_{\jhr}$. In general one has the following decomposition,
$$
[\Bu_{\ih}^{\jhr00\khr}u_{\khr},Z^J]u_{\jhr} =
\sum_{J_1+J_2= J\atop |J_2|\leq |J|-1}Z^{J_1}\big(\Bu_{\ih}^{\jhr00\khr}u_{\khr}\big)Z^{J_2}\delu_{00}u_{\khr}
+\Bu_{\ih}^{\jhr 00 \khr}u_{\khr}[\delu_{00},Z^J]u_{\jhr}
$$
By \eqref{proof used decay} and lemma \ref{pre lem one frame},
$$
\bigg|\sum_{J_1+J_2= J\atop |J_2|\leq |J|-1}Z^{J_1}\big(\Bu_{\ih}^{\jhr00\khr}u_{\khr}\big)Z^{J_2}\delu_{00}u_{\khr}\bigg|
\leq (Ks^2t^{-2})\cdot(CC_1\varepsilon t^{-3/2}s)\cdot \sum_{\kh\atop |J'|\leq |J|-1}|\del_{\alpha\beta}Z^{J'} u_{\kh}|.
$$
Similarly by \eqref{pre lem commutator second-order} the second term can by bounded as follows
$$
\big|\Bu_{\ih}^{\jhr 00 \khr}u_{\khr}[\delu_{00},Z^J]u_{\jhr}\big|
\leq (Ks^2t^{-2})\cdot(CC_1\varepsilon t^{-3/2}s)\cdot \sum_{\kh\atop |J'|\leq |J|-1}\big|\del_{\alpha\beta}Z^{J'}u_{\kh}\big|
$$
Notice that in the cone $\Lambda'$
\begin{equation}\label{proof estimate frame 2nd}
\big|\del_{\alpha\beta}u_{\ih}\big|\leq \sum_{\alpha,\beta}\big|\delu_{\alpha\beta} u_{\ih}\big|,
\end{equation}
by \eqref{proof used decay}, one has
$$
\big|[\Bu_{\ih}^{\jhr00\khr}u_{\khr},Z^J]u_{\jhr}\big|
\leq CK(C_1\varepsilon)^2t^{-3/2}s^{-1+2\delta} + CKC_1\varepsilon t^{-1/2}\cdot (s/t)^2\sum_{\kh\atop |J|\leq |I|-1}\big|\delu_{00}Z^Ju_{\kh}\big|.
$$
One concludes by
$$
\big|(s/t)^2\delu_{00}Z^I u_{\ih}\big|
\leq C(K+1)C_1\varepsilon t^{-3/2}s^{-1+2\delta} + CKC_1\varepsilon t^{-1/2}\cdot(s/t)^2\sum_{\kh\atop |J|\leq|I|-1}\big|\delu_{00}Z^Ju_{\kh}\big|,
$$
when $|J|=0$, the last term in right-hand-side disappears.
Then by induction, the first inequality is proved.

For the second inequality, one observes that it is a trivial result of the first inequality, \eqref{proof estimate frame 2nd} and
\eqref{pre lem commutator second-order}.
\end{proof}

At the end of this section, one will give the $L^2$ estimates of the ``bad" derivatives.
\begin{lemma}\label{proof lem 00 energy estimate}
Let $\{w_i\}$ be solution of \eqref{main eq main} and suppose that \eqref{proof used energy estimates}, \eqref{proof used decay} hold.
Then for any $|I|\leq 2$ the following estimates hold  with $C_1\varepsilon \leq 1$.
$$
\bigg(\int_{H_s}\big|s^3t^{-2}\delu_{00} Z^Iu_{\ih}\big|^2dx\bigg)^{1/2} \leq C(K+1)C_1\varepsilon,
$$
for any $|I^*|\leq 3$
$$
\bigg(\int_{H_s}\big|s^3t^{-2}\delu_{00} Z^{I^*}u_{\ih}\big|^2dx\bigg)^{1/2} \leq C(K+1)C_1\varepsilon s^{\delta}.
$$
\end{lemma}
\begin{proof}
Similar to that of lemma \ref{proof lem 00 decay}, the proof is mainly a $L^2$ estimate of $\mathcal {R}_{\ih}$.
Lemma \ref{proof lem L2 F 3} (lemma \ref{proof lem L2 F 4}) gives the $L^2$ estimate of $Z^IF_{\ih}$  (respectively $Z^{I^*} F_{\ih}$ ).

$\Gt_{\ih}^{\jr\alphar\betar}\del_{\alphar}\del_{\betar}w_{\jr}$ can be decomposed as follows:
$$
\Gt_{\ih}^{\jr\alphar\betar}\del_{\alphar}\del_{\betar}w_{\jr}
= A_{\ih}^{\jhr\alphar\betar\gammar\khr}\del_{\gammar}u_{\khr}\del_{\alphar}\del_{\betar}u_{\jhr}
+ A_{\ih}^{\jhr\alphar\betar\gammar\kcr}\del_{\gammar}v_{\kcr}\del_{\alphar}\del_{\betar}u_{\jhr}
+ B_{\ih}^{\jhr\alphar\betar\kcr}v_{\kcr}\del_{\alphar}\del_{\betar}u_{\jhr}
+ G_{\ih}^{\jcr\alphar\betar}\del_{\alphar}\del_{\betar}v_{\jcr}.
$$
The last three terms are finite linear combinations of $v_{\jc}\del_{\alpha}\del_{\beta}w_j,\,\del_{\alpha}w_j\del_{\beta}\del_{\gamma}v_{\jc}$ and
$\del_{\alpha}v_{\jc}\del_{\beta}\del_{\gamma}w_j$.  When $|I^*|\leq 3$,
\begin{equation}\label{proof lem 00 energy estimate eq1}
\aligned
&\bigg(\int_{H_s}\big|Z^{I^*}\big(v_{\jc}\del_{\alpha}\del_{\beta}w_j\big)\big|^2\bigg)^{1/2}
\\
&\leq \sum_{|I_1^*|\leq 2\atop I_1^*+I_2^*=I^*}
\bigg(\int_{H_s}\big|Z^{I_1^*}v_{\jc}Z^{I_2^*}\del_{\alpha}\del_{\beta}w_j\big)\big|^2\bigg)^{1/2}
+\bigg(\int_{H_s}\big|Z^{I^*}v_{\jc}\del_{\alpha}\del_{\beta}w_j\big|^2\bigg)^{1/2}
\\
&\leq CKC_1\varepsilon s^{-3/2+\delta}\bigg(\int_{H_s}\big|(s/t)Z^{I_2^*}\del_{\alpha}\del_{\beta}w_j\big|^2\bigg)^{1/2}
+ CKC_1\varepsilon  s^{-3/2+\delta}\bigg(\int_{H_s}\big|Z^{I^*}v_{\jc}\big|^2\bigg)^{1/2}
\\
&\leq CK(C_1\varepsilon)^2 s^{-3/2}.
\endaligned
\end{equation}
Similarly,
$$
\bigg(\int_{H_s}\big|Z^{I^*}\big(\del_{\alpha}w_j\del_{\beta}\del_{\gamma}v_{\jc}\big)\big|^2\bigg)^{1/2}
+\bigg(\int_{H_s}\big|Z^{I^*}\big(\del_{\alpha}v_{\jc}\del_{\beta}\del_{\gamma}w_j\big)\big|^2dx\bigg)^{1/2}
\leq CK(C_1\varepsilon)^2 s^{-3/2}.
$$

Of course, when $|I|\leq 2$,
$$
\aligned
\bigg(\int_{H_s}\big|Z^I(v_{\jc}\del_{\alpha}\del_{\beta}\del_{\gamma}w_j)\big|^2dx\bigg)^{1/2}
&+ \bigg(\int_{H_s}\big|Z^I(\del_{\alpha}v_{\jc}\del_{\beta}w_j)\big|^2dx\bigg)^{1/2}
\\
&+\bigg(\int_{H_s}\big|Z^I\big(\del_{\alpha}w_j\del_{\beta}\del_{\gamma}v_{\jc}\big)\big|^2dx\bigg)^{1/2}
\leq C(C_1\varepsilon)^2 s^{-3/2+2\delta},
\endaligned
$$
The term $A_{\ih}^{\jhr\alphar\betar\gammar\khr}\del_{\gammar}u_{\khr}\del_{\alphar}\del_{\betar}u_{\jhr}$ can be written under one-frame:
$$
A_{\ih}^{\jhr\alphar\betar\gammar\khr}\del_{\gammar}u_{\khr}\del_{\alphar}\del_{\betar}u_{\jhr}
=
 \Au_{\ih}^{\jhr\alphar\betar\gammar\khr}\delu_{\gammar}u_{\khr}\delu_{\alphar}\delu_{\betar}u_{\jhr}
-\Au_{\ih}^{\jhr\alphar\betar\gammar\khr}\delu_{\gammar}u_{\khr}(\del_{\alphar}\Phi_{\betar}^{\betar'})\del_{\betar'}u_{\jhr}
$$
Recall that for any multi-index $J$,
$\big|Z^J\delu_{\alpha}\Phi_{\beta}^{\beta'}\big| \leq C(J)t^{-1}$ and $\big|Z^J\Au_{\ih}^{\jh\alpha\beta\gamma\kh}\big|\leq C(J)K$.
One has for $|I^*|\leq 3$,
$$
\bigg(\int_{H_s}\big|Z^{I^*}\big(\Au_{\ih}^{\jhr\alphar\betar\gammar\khr}\delu_{\gammar}u_{\khr}(\del_{\alphar}\Phi_{\betar}^{\betar'})\del_{\betar'}u_{\jhr}\big)\big|^2dx\bigg)^{1/2}
\leq CK(C_1\varepsilon)^2 s^{-3/2}.
$$
Take the null conditions \eqref{main conditions null} into consideration, for any multi-index $I$,
$$
\big|Z^I\big(\Au_{\ih}^{\jh000\kh}\big)\big|\leq C(I)(s/t)^2.
$$

Then,
$$
\aligned
&\quad\quad\big|Z^I\big(\Au_{\ih}^{\jhr\alphar\betar\gammar\khr}\delu_{\gammar}u_{\khr}\delu_{\alphar}\delu_{\betar}u_{\jhr}\big)\big|
\\
&\leq \sum_{\kh,\jh\atop I_1+I_2+I_3 = I} \big|Z^{I_3}\Au_{\ih}^{\jh000\kh}\big|\big|Z^{I_1}\delu_0 u_{\kh}\big|\big|Z^{I_2}\delu_0\delu_0 u_{\jh}\big|
+     \sum_{\kh,\jh,a,\beta,\gamma\atop I_1+I_2+I_3 = I}\big|Z^{I_3}\Au_{\ih}^{\jh a\beta\gamma\kh}\big|\big|Z^{I_1}\delu_{\gamma} u_{\kh}\big|\big|Z^{I_2}\delu_a\delu_{\beta} u_{\jh}\big|
\\
&\quad + \sum_{\kh,\jh,\alpha,\beta,c\atop I_1+I_2+I_3 = I}\big|Z^{I_3}\Au_{\ih}^{\jh \alpha\beta c\kh}\big|\big|Z^{I_1}\delu_c u_{\kh}\big|\big|Z^{I_2}\delu_{\alpha}\delu_{\beta} u_{\jh}\big|
+        \sum_{\kh,\jh,\alpha,b,\gamma\atop I_1+I_2+I_3 = I}\big|Z^{I_3}\Au_{\ih}^{\jh \alpha b\gamma\kh}\big|\big|Z^{I_1}\delu_{\gamma} u_{\kh}\big|\big|Z^{I_2}\delu_{\alpha}\delu_b u_{\jh}\big|
\\
&\leq CK\sum_{\kh,\jh\atop I_1+I_2+I_3 = I}  (s/t)^2\big|Z^{I_1}\delu_0 u_{\kh}\big|\big|Z^{I_2}\delu_0\delu_0 u_{\jh}\big|
+     CK\sum_{\kh,\jh,a,\beta,\gamma\atop I_1+I_2+I_3 = I}\big|Z^{I_1}\delu_{\gamma} u_{\kh}\big|\big|Z^{I_2}\delu_a\delu_{\beta} u_{\jh}\big|
\\
&\quad + CK\sum_{\kh,\jh,\alpha,\beta,c\atop I_1+I_2+I_3 = I}\big|Z^{I_1}\delu_c u_{\kh}\big|\big|Z^{I_2}\delu_{\alpha}\delu_{\beta} u_{\jh}\big|
+        CK\sum_{\kh,\jh,\alpha,b,\gamma\atop I_1+I_2+I_3 = I}\big|Z^{I_1}\delu_{\gamma} u_{\kh}\big|\big|Z^{I_2}\delu_{\alpha}\delu_b u_{\jh}\big|
\\
&=: M_0 + M_1 + M_2 + M_3.
\endaligned
$$
By the same argument of \eqref{proof lem 00 energy estimate eq1}, one has
$$
\sum_{k=0}^3\bigg(\int_{H_s}\big|M_k\big|^2dx\bigg)^{1/2}\leq C(C_1\varepsilon)^2s^{-3/2 + 2\delta},\quad \text{when} |I|\leq 2,
$$
and
$$
\sum_{k=0}^3\bigg(\int_{H_s}\big|M_k\big|^2dx\bigg)^{1/2}\leq C(C_1\varepsilon)^2s^{-1+\delta},\quad \text{when} |I|\leq 3.
$$
So one concludes by
$$
\bigg(\int_{H_s}\big|Z^{I^*}\big(\Gt_{\ih}^{\jr\alphar\betar}\del_{\alphar\betar}w_{\jr}\big)\big|^2dx\bigg)^{1/2}
\leq KC(C_1\varepsilon)^2s^{-1+\delta}, \quad \text{for} |I^*|\leq 3
$$
and
$$
\bigg(\int_{H_s}\big|Z^I\big(\Gt_{\ih}^{\jr\alphar\betar}\del_{\alphar\betar}w_{\jr}\big)\big|^2dx\bigg)^{1/2} \leq KC(C_1\varepsilon)^2s^{-3/2+2\delta} ,
\quad \text{for} |I|\leq 2.
$$

Recall that $\big|Z^J\delu_{\alpha}\Phi_{\beta}^{\beta'}\big| \leq Ct^{-1}$ and $\big|Z^J\Bu_{\ih}^{\jh\alpha\beta\kh}\big|\leq K$. One has for
$|I^*|\leq 3$.
$$
\bigg(\int_{H_s}\big|Z^{I^*}\big(\Bu_{\ih}^{\jhr\alphar\betar\khr}u_{\khr}\delu_{\alphar}\Phi_{\betar}^{\betar'}\del_{\betar'}u_{\jhr}\big)\big|^2dx\bigg)^{1/2}
\leq C(C_1\varepsilon)^2 s^{-3/2}.
$$
Similarly, for $|I^*|\leq 3$,
$$
\bigg(\int_{H_s}\big|\underline{m}^{\alphar\betar}(\delu_{\alphar}\Phi_{\betar}^{\betar'})\del_{\betar'}Z^{I^*}u_{\ih}\big|^2dx\bigg)
\leq CC_1\varepsilon s^{-1+\delta}
$$
and for $|I|\leq 2$
$$
\bigg(\int_{H_s}\big|\underline{m}^{\alphar\betar}(\delu_{\alphar}\Phi_{\betar}^{\betar'})\del_{\betar'}Z^Iu_{\ih}\big|^2dx\bigg)
\leq CC_1\varepsilon s^{-1}.
$$

Now one turns to the term $[\Bu_{\ih}^{\jhr00\khr}u_{\khr}\delu_{\alphar\betar},Z^I]u_{\jhr}$.
Recall the following decomposition
\begin{equation}\label{proof lem 00 energy estimate eq1}
[\Bu_{\ih}^{\jhr00\khr}u_{\khr},Z^I]u_{\jhr} =
\sum_{I_1+I_2= I\atop |I_2|\leq |I|-1}Z^{I_1}\big(\Bu_{\ih}^{\jhr00\khr}u_{\khr}\big)Z^{I_2}\delu_{00}u_{\jhr}
+\Bu_{\ih}^{\jhr 00 \khr}u_{\khr}[\delu_{00},Z^I]u_{\jhr}
\end{equation}

Notice that
$$
\aligned
\big|[\delu_{00},Z^I]u_{\jh}\big|
&\leq C\sum_{\alpha,\beta\atop |J|\leq |I|-1}|\del_{\alpha\beta}Z^Ju_{\jh}|
\leq C\sum_{\alpha,\beta\atop |J|\leq |I|-1}|\delu_{\alpha\beta}Z^Ju_{\jh}|
\\
&\leq C\sum_{|J|\leq |I|-1}|\delu_{00}Z^Ju_{\jh}| + C\sum_{a,\beta\atop |J|\leq |I|-1}|\delu_{a\beta}Z^Ju_{\jh}|,
\endaligned
$$

The first term in right-hand-side of \eqref{proof lem 00 energy estimate eq1} is estimated as follows: when $|I|\leq 3$,
$$
\aligned
&\sum_{I_1+I_2= I\atop |I_2|\leq |I|-1}\bigg(\int_{H_s}\big|Z^{I_1}\big(\Bu_{\ih}^{\jhr00\khr}u_{\khr}\big)Z^{I_2}\delu_{00}u_{\jhr}\big|^2dx\bigg)^{1/2}
\\
&\leq \sum_{I_1+I_2= I\atop |I_2|\leq |I|-1}\bigg(\int_{H_s}\big|Z^{I_1}\big(\Bu_{\ih}^{\jhr00\khr}u_{\khr}\big)\delu_{00}Z^{I_2}u_{\jhr}\big|^2dx\bigg)^{1/2}
\\
&\quad+\sum_{I_1+I_2= I\atop |I_2|\leq |I|-1,a,\beta}\bigg(\int_{H_s}\big|Z^{I_1}\big(\Bu_{\ih}^{\jhr00\khr}u_{\khr}\big)\delu_{a\beta}Z^{I_2}u_{\jhr}\big|^2dx\bigg)^{1/2}
\endaligned
$$
For the first term:
$$
\aligned
&\sum_{I_1+I_2= I\atop |I_2|\leq |I|-1}\bigg(\int_{H_s}\big|Z^{I_1}\big(\Bu_{\ih}^{\jhr00\khr}u_{\khr}\big)\delu_{00}Z^{I_2}u_{\jhr}\big|^2dx\bigg)^{1/2}
\\
&\leq \sum_{I_1+I_2= I\atop |I_1|=1}+\sum_{I_1+I_2= I\atop |I_2|\leq 1}
 \bigg(\int_{H_s}\big|Z^{I_1}\big(\Bu_{\ih}^{\jhr00\khr}u_{\khr}\big)\delu_{00}Z^{I_2}u_{\jhr}\big|^2dx\bigg)^{1/2}
\\
&\leq \sum_{I_1+I_2= I\atop \jh,|I_1|= 1}\bigg(\int_{H_s}\big|CKC_1\varepsilon t^{-3/2} \cdot s(s/t)^2\delu_{00}Z^{I_2}u_{\jhr}\big|^2dx\bigg)^{1/2}
\\
&\quad +\sum_{I_1+I_2= I\atop \kh,|I_2|\leq 1}\bigg(\int_{H_s}\big|CKC_1\varepsilon t^{-3/2}s^{-1+2\delta}Z^{I_1}u_{\kh}\big|^2dx\bigg)^{1/2}
\\
&\leq CK(C_1\varepsilon)^2s^{-3/2+2\delta} + CKC_1\varepsilon s^{-3/2} \sum_{\jh\atop|J|\leq |I|-1}\bigg(\int_{H_s}\big|s^3t^{-2}\delu_{00}Z^J u_{\jh}\big|\bigg)^{1/2}.
\endaligned
$$

The second term in right-had-side of \eqref{proof lem 00 energy estimate eq1} is estimated as follows: when $|I|\leq 3$,
$$
\aligned
&\quad \bigg(\int_{H_s}\big|\Bu_{\ih}^{\jhr 00 \khr}u_{\khr}[\delu_{00},Z^I]u_{\jhr}\big|^2dx\bigg)^{1/2}
\\
&\leq CKC_1\varepsilon \sum_{a,\beta,\jh\atop |J|\leq |I|-1}\bigg(\int_{H_s}t^{-3}\big|s\delu_{a\beta}Z^Ju_{\jh}\big|^2dx\bigg)^{1/2}
     + CKC_1\varepsilon s^{-3/2}\sum_{\kh\atop |J|\leq |I|-1}\bigg(\int_{H_s}|s^3t^{-2}\delu_{00} Z^J u_{\kh}|^2dx\bigg)^{1/2}
\\
&\leq CKC_1\varepsilon s^{-3/2} \sum_{\jh\atop |J|\leq |I|}E(s,Z^Ju_{\jh})^{1/2}
+  CKC_1\varepsilon s^{-3/2}\sum_{\kh\atop |J|\leq |I|-1}\bigg(\int_{H_s}|s^3t^{-2}\delu_{00}Z^Ju_{\kh}|^2dx\bigg)^{1/2}
\\
&\leq CK(C_1\varepsilon)^2 s^{-3/2+\delta} + CKC_1\varepsilon s^{-3/2}\sum_{\kh\atop |J|\leq |I|-1}\bigg(\int_{H_s}|s^3 t^{-2}\delu_{00}Z^Ju_{\kh}|^2dx\bigg)^{1/2}
\endaligned
$$
To estimate the only linear terms,
$$
\underline{m}^{\ar\br}\delu_{\ar\br}Z^I u_{\ih} +\underline{m}^{\ar0}\delu_{\ar0}Z^Iu_{\ih} + \underline{m}^{0\ar}\del_{0\ar}Z^Iu_{\ih}
$$
one notice that they are ``good derivatives". By using directly the last inequality of \eqref{proof used decay}, one has when $|I|\leq 2$,
$$
\bigg(\int_{H_s}\big|\underline{m}^{\ar\betar}\delu_{\ar\betar}Z^Iu_{\ihr}\big|^2dx\bigg)^{1/2}\leq CC_1\varepsilon s^{-1}.
$$
When $|I^*|\leq 3$,
$$
\bigg(\int_{H_s}\big|\underline{m}^{\ar\betar}\delu_{\ar\betar}Z^{I^*}u_{\ihr}\big|^2dx\bigg)^{1/2}\leq CC_1\varepsilon s^{-1+\delta}.
$$
So finally one gets, when $|I|\leq 2$
$$
\aligned
&\quad\bigg(\int_{H_s}\big|s^2t^{-2}\delu_{00} Z^Iu_{\ih}\big|^2dx\bigg)^{1/2}
\\
&\leq \max_{\jh}||\mathcal {R}_{\jh}||_{L^2(H_s)}
\\
&\leq CK(C_1\varepsilon)^2s^{-3/2+2\delta} + CC_1\varepsilon s^{-1}
+  CKC_1\varepsilon s^{-3/2} \sum_{\alpha,\beta,\kh\atop |J|\leq |I|-1}\bigg(\int_{H_s}|s^3t^{-2}\delu_{00}u_{\kh}|^2dx\bigg)^{1/2}.
\endaligned
$$
By induction, one gets the desired result.
The case where $|I^*|\leq 3$ can be proved similarly, one omits the details.
\end{proof}
\subsection{Part four -- Estimates of other source terms}
\begin{lemma}\label{proof lem L2 com complete}
Let $\{w_i\}$ be regular solution of \eqref{main eq main}. Suppose that \eqref{proof used energy estimates} and \eqref{proof used decay},
then for any $|I^*|\leq 4$, the following estimate holds:
$$
\bigg(\int_{H_s}\big|[G_i^{\jr\alphar\betar}\del_{\alphar}\del_{\betar},Z^{I^*}]w_{\jr}\big|^2dx\bigg)^{1/2}
\leq C(C_1\varepsilon)^2 K s^{-1+\delta}.
$$
For any $|I|\leq 3$ the following estimate holds:
$$
\bigg(\int_{H_s}\big|[G_{\ih}^{\jr\alphar\betar}\del_{\alphar}\del_{\betar},Z^I]w_{\jr}\big|^2dx\bigg)^{1/2}
\leq C(C_1\varepsilon)^2 K s^{-3/2+2\delta}.
$$
\end{lemma}
\begin{proof}
By definition,
$$
[G_i^{\jr\alphar\betar}\del_{\alphar}\del_{\betar},Z^{I^*}]w_{\jr}
= [\Gt_i^{\jr\alphar\betar}\del_{\alphar}\del_{\betar},Z^{I^*}]w_{\jr} + [B_i^{\jhr\alphar\betar}\del_{\alphar}\del_{\betar},Z^{I^*}]u_{\jhr}
$$
The first component is controlled by lemma \ref{proof lem L2 com 4}. The second term is estimated similarly. First one rewrite it under one-frame.
\begin{equation}\label{proof lem L2 com complete eq1}
[B_i^{\jhr\alphar\betar\khr}u_{\khr}\del_{\alphar}\del_{\betar},Z^{I^*}]u_{\jhr}
= [\Bu_i^{\jhr\alphar\betar\khr}u_{\khr}\delu_{\alphar}\delu_{\betar},Z^{I^*}]u_{\jhr}
+ [\Bu_i^{\jhr\alphar\betar\khr}u_{\khr}(\delu_{\alphar}\Phi_{\betar}^{\betar'})\delu_{\betar'},Z^{I^*}]u_{\jhr}
\end{equation}

One recalls that
$$
\big|Z^I\del_{\alpha}\Phi_{\alpha}^{\beta}\big|\leq Ct^{-1}.
$$
With this additional decay, as in the proof of lemma \ref{proof lem com 3}, the $L^2$ norm of the second term in right-hand-side of
\eqref{proof lem L2 com complete eq1} on $H_s$ is bounded by $C(C_1\varepsilon)^2s^{-3/2+\delta}$.

One notices that $B_i^{j\alphar\betar\khr}\xi_{\alphar}\xi_{\betar}$ is a null form so following the lemma \ref{pre lem one frame} one has
$$
\big|Z^I\Bu_i^{j00\khr}\big|\leq C(I)(s/t)^2,
$$
Then exactly as in the proof of lemma \ref{proof lem com 3} and take the estimates given by
lemma \ref{proof lem 00 decay} and lemma \ref{proof lem 00 energy estimate} into consideration, the $L^2$ norm of the first term in right-hand-side
of \eqref{proof lem L2 com complete eq1} on $H_s$ is bounded by $C(C_1\varepsilon)^2s^{-3/2+\delta}$.

The proof of the second estimate is the same. One omits the details.
\end{proof}
\begin{lemma}\label{proof lem Mv 4}
\label{quasilinear lem energy curveterm is small}
Suppose \eqref{proof used decay} and \eqref{proof used energy estimates} hold,
then for any $|I^*|\leq 4$ the following estimates is true:
\begin{equation}
\aligned
&\bigg|\int_{H_s}\frac{s}{t}\bigg(\big(\del_{\alphar}G_i^{\jr\alphar\betar}\big)\del_t Z^{I^*} w_i \del_{\beta}Z^{I^*} w_{\jr}
- \frac{1}{2}\big(\del_t G_i^{\jr\alphar\betar}\big)\del_{\alphar}Z^{I^*} w_i \del_{\betar}Z^{I^*} w_{\jr}\bigg)dx\bigg|
\\
&\quad\leq CC_1\varepsilon s^{-1+\delta}E_m(s,Z^{I^*} w_i)^{1/2}.
\endaligned
\end{equation}
\end{lemma}
\begin{proof}
The proof is mainly a substitution of \eqref{proof used decay} and \eqref{proof used energy estimates}. One writes the estimate of
$$\bigg|\int_{H_s}(s/t)\big(\del_{\alphar}G_i^{\jr\alphar\betar}\big)\del_t Z^{I^*} w_i \del_{\betar}Z^{I^*} w_{\jr} dx \bigg|
$$
in detail and omits the rest part.
First one notices that
$$
\big|\del_{\alpha}G_i^{j\alpha\beta}\big|\leq C\sum_j|\del_{\alpha}w_j| + C\sum_{j,\beta}|\del_{\alpha}\del_{\beta}w_j|\leq C(C_1\varepsilon)^2t^{-1/2}s^{-1}.
$$
Substitute this into the expression,
$$
\aligned
&\quad\bigg|\int_{H_s}(s/t)\big(\del_{\alphar}G_i^{\jr\alphar\betar}\big)\del_t Z^{I^*} w_i \del_{\betar}Z^{I^*} w_{\jr} dx\bigg|
\\
&\leq \bigg|\int_{H_s}\big((t/s)\del_{\alphar}G_i^{\jr\alphar\betar}\big)(s/t)\del_t Z^{I^*} w_i (s/t)\del_{\betar}Z^{I^*} w_{\jr}dx\bigg|
\\
&\leq \sum_{j,\beta}\int_{H_s}C(C_1\varepsilon)t^{1/2}s^{-2}\big|(s/t)\del_t Z^{I^*} w_i (s/t)\del_{\beta}Z^{I^*} w_j\big|dx
\\
&\leq C(C_1\varepsilon)^2s^{-1}\sum_{j,\beta}\int_{H_s}\big|(s/t)\del_t Z^{I^*} w_i (s/t)\del_{\beta}Z^{I^*} w_j\big|dx
\\
&\leq
C(C_1\varepsilon)^2s^{-1}\sum_{j,\beta}\bigg(\int_{H_s}\big|(s/t)\del_{\beta}Z^{I^*} w_j\big|^2dx\bigg)^{1/2} \cdot \bigg(\int_{H_s}\big|(s/t)\del_t Z^{I^*} w_i \big|^2dx\bigg)^{1/2}
\\
&\leq C(C_1\varepsilon)^2 s^{-1+\delta}\bigg(\int_{H_s}\big|(s/t)\del_t Z^{I^*} w_i \big|^2dx\bigg)^{1/2}
\\
&\leq C(C_1\varepsilon)^2 s^{-1+\delta}E_m(s,Z^{I^*}w_i)^{1/2}.
\endaligned
$$
\end{proof}
\begin{lemma}\label{proof lem Mv 3}
Suppose \eqref{proof used decay} and \eqref{proof used energy estimates} hold,
then for any $|I|\leq 3$ the following estimates is true:
\begin{equation}
\aligned
&\bigg|\int_{H_s}\frac{s}{t}\bigg(\big(\del_{\alphar}G_{\ih}^{j\alphar\betar}\big)\del_t Z^I u_{\ih} \del_{\betar}Z^I w_{\jr}
- \frac{1}{2}\big(\del_t G_{\ih}^{\jr\alphar\betar}\big)\del_{\alphar}Z^{I^*} u_{\ih} \del_{\betar}Z^{I^*} w_{\jr}\bigg)dx\bigg|
\\
&\quad\leq CC_1\varepsilon s^{-3/2+2\delta}E_m(s,Z^{I^*} u_{\ih})^{1/2}.
\endaligned
\end{equation}
\end{lemma}
\begin{proof}
Again one will only write the estimate on $\big(\del_{\alpha}G_{\ih}^{j\alpha\beta}\big)\del_t Z^I u_{\ih} \del_{\beta}Z^I w_j$ in detail. By definition,
\begin{equation}\label{proof lem Mv 3 eq1}
\big(\del_{\alphar}G_{\ih}^{\jr\alphar\betar}\big)\del_t Z^I u_{\ih} \del_{\beta}Z^I w_{\jr}
= \big(\del_{\alphar}G_{\ih}^{\jcr\alphar\betar}\big)\del_t Z^I u_{\ih} \del_{\betar}Z^I v_{\jcr}
+ \big(\del_{\alphar}G_{\ih}^{\jhr\alphar\betar}\big)\del_t Z^I u_{\ih} \del_{\betar}Z^I u_{\jhr}.
\end{equation}
The second term in right-hand-side is decomposed again as follows:
\begin{align}\label{proof lem Mv 3 eq2}
\big(\del_{\alphar}G_{\ih}^{\jhr\alphar\betar}\big)\del_t Z^I u_{\ih} \del_{\betar}Z^I u_{\jhr}
=& \del_{\alphar}\big(A_{\ih}^{\jhr\alphar\betar\kcr}\del_{\gamma}v_{\kcr}
+ B_{\ih}^{\jhr\alphar\betar\kcr}v_{\kcr}\big)\del_t Z^I u_{\ih} \del_{\betar}Z^I u_{\jhr}
\\
+& \del_{\alphar}\big(A_{\ih}^{\jhr\alphar\betar\khr}\del_{\gamma}u_{\khr}
+ B_{\ih}^{\jhr\alphar\betar\khr}u_{\khr}\big)\del_t Z^I u_{\ih} \del_{\betar}Z^I u_{\jhr}\nonumber.
\end{align}
The estimates of the first term in right-hand-side of \eqref{proof lem Mv 3 eq1} and the first term in right-hand-side of \eqref{proof lem Mv 3 eq2}
is simpler.
$$
\aligned
&\quad\bigg|\int_{H_s}(s/t)\big(\del_{\alphar}G_{\ih}^{\jcr\alphar\betar}\big)\del_t Z^I u_{\ih} \del_{\betar}Z^I v_{\jcr}dx\bigg|
\\
&\leq \bigg|\int_{H_s}CKC_1\varepsilon t^{-1/2}s^{-1} (s/t)\del_t Z^I u_{\ih} \del_{\betar}Z^I v_{\jcr}dx\bigg|
\\
&\leq CKC_1\varepsilon s^{-3/2} \sum_{\beta,\jc}\bigg(\int_{H_s}\big| (s/t)\del_t Z^I u_{\ih}\big|^2dx\bigg)^{1/2}\cdot\bigg(\int_{H_s}\big|\del_{\beta}Z^I v_{\jc}\big|^2dx\bigg)^{1/2}
\\
&\leq CK(C_1\varepsilon)^2 s^{-3/2+2\delta}\bigg(\int_{H_s}\big| (s/t)\del_t Z^I u_{\ih}\big|^2dx\bigg)^{1/2}
\\
&\leq CK(C_1\varepsilon)^2 s^{-3/2+2\delta}E_m(s,Z^I u_{\ih})^{1/2}.
\endaligned
$$

$$
\aligned
&\quad\bigg|\int_{H_s}(s/t)\del_{\alphar}\big(A_{\ih}^{\jhr\alphar\betar\kcr}\del_{\gamma}v_{\kcr}
+ B_{\ih}^{\jhr\alphar\betar\kcr}v_{\kcr}\big)\del_t Z^I u_{\ih} \del_{\betar}Z^I u_{\jhr}dx\bigg|
\\
&\leq \sum_{\beta,\jh}|\int_{H_s}KCC_1\varepsilon t^{-3/2+2\delta}(t/s)\cdot (s/t)\del_t Z^I u_{\ih} (s/t)\del_{\beta}Z^I u_{\jh}dx\bigg|
\\
&\leq \sum_{\beta,\jh}\int_{H_s}KCC_1\varepsilon t^{-1/2+2\delta}s^{-1}\cdot \big|(s/t)\del_t Z^I u_{\ih} (s/t)\del_{\beta}Z^I u_{\jh}\big|dx
\\
&\leq C(C_1\varepsilon)^2 s^{-3/2+2\delta}\bigg(\int_{H_s}\big|(s/t)\del_t Z^I u_{\ih}\big|^2dx\bigg)^{1/2}
\\
&\leq C(C_1\varepsilon)^2 s^{-3/2+2\delta}E_m(s,Z^I u_{\ih})^{1/2}.
\endaligned
$$

To estimate the last term of \eqref{proof lem Mv 3 eq2}, one needs the null condition \eqref{main conditions null}. First, one writes this term under
one-frame:
\begin{equation}\label{proof lem Mv 3 eq3}
\aligned
&\del_{\alphar}\big(A_{\ih}^{\jhr\alphar\betar\khr}\del_{\gamma}u_{\khr}
+ B_{\ih}^{\jhr\alphar\betar\khr}u_{\khr}\big)\del_t Z^I u_{\ih} \del_{\betar}Z^I u_{\jhr}
\\
&=\delu_{\alphar}\big(\Au_{\ih}^{\jhr\alphar\betar\khr}\delu_{\gamma}u_{\khr}
+ \Bu_{\ih}^{\jhr\alphar\betar\khr}u_{\khr}\big)\del_t Z^I u_{\ih} \delu_{\betar}Z^I u_{\jhr}
\\
&\quad+ \del_{\alphar'}\big(\Phi_{\alphar}^{\alphar'}\Phi_{\betar}^{\betar'}\big)\big(\Au_{\ih}^{\jhr\alphar\betar\khr}\delu_{\gamma}u_{\khr}
+ \Bu_{\ih}^{\jhr\alphar\betar\khr}u_{\khr}\big)\del_t Z^I u_{\ih} \del_{\betar'}Z^I u_{\jhr}.
\endaligned
\end{equation}
Taking into account the fact that $\big|\del_{\gamma}\big(\Phi_{\alpha}^{\alpha'}\Phi_{\beta}^{\beta'}\big)\big|\leq Ct^{-1}$ as in the proof of
lemma \ref{proof lem com 3}, one can easily proof that
$$
\aligned
&\bigg|\int_{H_s}(s/t)\del_{\alphar'}\big(\Phi_{\alphar}^{\alphar'}\Phi_{\betar}^{\betar'}\big)\big(\Au_{\ih}^{\jhr\alphar\betar\khr}\delu_{\gamma}u_{\khr}
+ \Bu_{\ih}^{\jhr\alphar\betar\khr}u_{\khr}\big)\del_t Z^I u_{\ih} \del_{\betar'}Z^I u_{\jhr}dx\bigg|
\\
&\leq C(C_1\varepsilon)^2s^{-3/2}E_m(s,Z^Iu_{\ih})^{1/2}.
\endaligned
$$

Now the most difficult term, the first term in right-hand-side of \eqref{proof lem Mv 3 eq3} will be considered.
$$
\aligned
&\quad\delu_{\alphar}\big(\Au_{\ih}^{\jhr\alphar\betar\gammar\khr}\delu_{\gammar}u_{\khr}
+ \Bu_{\ih}^{\jhr\alphar\betar\khr}u_{\khr}\big)\cdot\del_t Z^I u_{\ih} \delu_{\betar}Z^I u_{\jhr}
\\
&= \delu_0\big(\Au_{\ih}^{\jhr000\khr}\delu_0u_{\khr}\delu_0Z^Iu_{\jhr}
+ \Bu_{\ih}^{\jhr00\khr}u_{\khr}\delu_0Z^Iu_{\jhr}\del_tu_{\ih}\big)\cdot \del_tZ^I u_{\ih}
\\
&\quad+\delu_0\big(\Au_{\ih}^{\jhr0\br0\khr}\delu_0u_{\khr}\delu_{\br}Z^Iu_{\jhr} + \Au_{\ih}^{\jhr00\cred\khr}\delu_{\cred}u_{\khr}\delu_{\br}Z^Iu_{\jhr}
+\Au_{\ih}^{\jhr0\br\cred\khr}\delu_{\cred}u_{\khr}\delu_{\br}Z^Iu_{\jhr}\big)\cdot \del_tZ^I u_{\ih}
\\
&\quad+ \delu_0\big(\Bu_{\ih}^{\jhr0\br\khr}u_{\khr}\delu_{\br}Z^Iu_{\jhr}\del_tu_{\ih}\big)\cdot \del_tZ^I u_{\ih}
\\
&\quad+ \delu_{\ar}\big(\Au_{\ih}^{\jhr\ar\betar\gammar\khr}\delu_{\gammar}u_{\khr}
+ \Bu_{\ih}^{\jhr\ar\betar\khr}u_{\khr}\big)\del_t Z^I u_{\ih} \delu_{\betar}Z^I u_{\jhr}
\\
&:= \delu_0\big(\Au_{\ih}^{\jhr000\khr}\delu_0u_{\khr}\delu_0Z^Iu_{\jhr}
+ \Bu_{\ih}^{\jhr00\khr}u_{\khr}\delu_0Z^Iu_{\jhr}\del_tu_{\ih}\big)\cdot \del_tZ^I u_{\ih} + \mathscr{N}\cdot \del_t Z^I u_{\ih}.
\endaligned
$$
Notice that $\mathscr{N}$ is a linear combination of \\
$\Gamma \del_a\del_{\beta}u_{\ih}\del_{\gamma}Z^Iu_{\jh}\del_t u_{\kh}$,
$\Gamma \del_{\alpha}\del_b u_{\ih}\del_{\gamma}Z^Iu_{\jh}\del_t u_{\kh}$,
$\Gamma \del_{\alpha}\del_{\beta}u_{\ih}\del_c Z^Iu_{\jh}\del_t u_{\kh}$,
$\Gamma \del_{\beta}u_{\ih}\del_a\del_{\gamma}Z^Iu_{\jh}\del_t u_{\kh}$,
\\
$\Gamma \del_a\del_{\beta}u_{\ih}\del_{\gamma}Z^Iu_{\jh}\del_t u_{\kh}$,
$\Gamma \del_{\alpha}\del_bu_{\ih}\del_{\gamma}Z^Iu_{\jh}\del_t u_{\kh}$
and
$\Gamma \del_{\alpha}\del_{\beta}u_{\ih}\del_cZ^Iu_{\jh}\del_t u_{\kh}$ with $\Gamma$ a function bounded by $CK$.
By \eqref{proof used decay} and \eqref{proof used energy estimates 2-order}, one sees easily that
$$
\bigg|\int_{H_s}(s/t)\mathscr{N}\cdot \del_t Z^I u_{\ih}dx \bigg|\leq C(C_1\varepsilon)^2s^{-3/2+2\delta}E_m(s,Z^Iu_{\ih})^{1/2}.
$$

Taking into account the null condition \eqref{main conditions null}, one has by lemma \ref{pre lem one frame}:
$$
\big|\Au_{\ih}^{\jh000\kh}\big| + \big|\Bu_{\ih}^{\jh00\kh}\big| \leq C(s/t)^2.
$$
Then by lemma \ref{proof lem 00 decay} and lemma \ref{proof lem 00 energy estimate}, one can show that
$$
\bigg|\int_{H_s}(s/t)\delu_0\big(\Au_{\ih}^{\jhr000\khr}\delu_0u_{\khr}\delu_0Z^Iu_{\jhr}
+ \Bu_{\ih}^{\jhr00\khr}u_{\khr}\delu_0Z^Iu_{\jhr}\del_tu_{\ih}\big)\cdot \del_tZ^I u_{\ih}\bigg| \leq C(C_1\varepsilon)^2s^{-3/2+\delta}E_m(s,Z^Iu_{\ih})^{1/2}.
$$
So finally the desired result is proved.
\end{proof}
\subsection{Last Part -- Bootstrap argument}
First one verifies \eqref{pre lem energy curved energy is big} by the following lemma:
\begin{lemma}\label{proof lem curved energy is big}
Suppose \eqref{proof used decay} holds with $KC_1\varepsilon$ small enough. Then following estimate holds
$$
\sum_iE_g(s,Z^I w_i) \leq 3\sum_iE_m(s,Z^I w_i).
$$
\end{lemma}
\begin{proof}
One notice that
$$
\sum_{i,j,\alpha,\beta} \big|G_i^{j\alpha\beta}\big| \leq CK\sum_i\big(|\del w_i| + |w_i|\big).
$$

Then by simple calculation

$$
\aligned
&\quad \sum_i\big|E_G(s,w_i) - E_m(s,w_i) \big|
\\
&=  \bigg|2\int_{H_s} \big(\del_t w_i \del_{\beta}w_j G_i^{j\alpha\beta}\big) \cdot (1,-x^a/t) dx
- \int_{H_s} \big(\del_{\alpha}w_i\del_{\beta}w_j G_i^{j\alpha\beta}\big)dx\bigg|
\\
&\leq 2\int_{H_s}\bigg(\sum_{i,j,\alpha,\beta}\big|G_i^{j\alpha\beta}\big|\bigg)\cdot\bigg(\sum_{\alpha,k}|\del_{\alpha}w_k|^2\bigg) dx
\\
&\leq 2CK\int_{H_s}\sum_i\big(|\del w_i| + |w_i|\big)\cdot\bigg(\sum_{\alpha,k}|\del_{\alpha}w_k|^2\bigg)dx
\\
&\leq 2CKC_1\varepsilon \int_{H_s}\big(t^{-3/2}s^{\delta} + t^{-1/2}s^{-1} + t^{-3/2}s\big)(t/s)^2\cdot\bigg(\sum_{\alpha,k}|(s/t)\del_{\alpha}w_k|^2\bigg)dx
\\
&=    2CKC_1\varepsilon \int_{H_s}\big(t^{1/2}s^{-2+\delta} + t^{3/2}s^{-3} + t^{1/2}s^{-1}\big)\cdot\bigg(\sum_{\alpha,k}|(s/t)\del_{\alpha}w_k|^2\bigg)dx
\\
&\leq CKC_1\varepsilon \sum_iE_m(s,w_i).
\endaligned
$$
Here one takes $CKC_1\varepsilon \leq 2/3$ with $C$ a universal constant, then the lemma is proved.
\end{proof}

With all those preparations above, one is now ready to prove the main theorem.
\begin{proof}[Proof of theorem \ref{main thm main}]
Suppose that $\{w_i\}$ is the unique regular local-in-time solution of \eqref{main eq main}. Let $C_1,\,\varepsilon$ be positive constants. Suppose that
$[B+1,T^*]$ is the largest interval containing $B+1$ on which \eqref{proof energy assumption} holds for any $B+1\leq s\leq T^*$.
As one discussed in section 4.1,
there always exists an $\varepsilon'$ such that, when
$E^*_G(B+1,Z^{I^*}w_j)^{1/2}\leq \varepsilon'$,
$E_G(B+1,Z^{I^*}w_j)^{1/2}\leq (1/4)C_1\varepsilon$.
So by continuity $T^*>0$
when $\varepsilon'$ sufficiently small.

Also, if $T^*< +\infty$, then when $s=T^*$, at lest one of the three inequalities of \eqref{proof energy assumption}
will be replaced by a equality. That is as least on of the following equations holds:
\begin{equation}\label{proof energy assumption last}
\aligned
&E_m(T^*,Z^{I^*} v_{\jc})^{1/2} = C_1\varepsilon s^{\delta}, \quad \text{for}\quad j_0+1\leq \jc\leq n_0,\quad 0\leq |I^*|\leq 4,
\\
&E_m(T^*,Z^{I^*} u_{\ih})^{1/2} = C_1\varepsilon s^{\delta},\quad \text{for}\quad 1\leq \ih\leq j_0, \quad|I^*|= 4,
\\
&E_m(T^*,Z^I u_{\ih})^{1/2} = C_1\varepsilon,\quad \text{for}\quad 1\leq \ih\leq j_0, \quad |I|\leq 3.
\endaligned
\end{equation}
One derives the equation \eqref{main eq main} with respect to $Z^{I^*}$ with $|I^*|\leq 4$:
$$
\Box Z^{I^*}w_i + G_i^{\jr\alphar\betar}\del_{\alphar}\del_{\betar}Z^{I^*}w_{\jr} + D_i^2 w_i = [G_i^{\jr\alphar\betar}\del_{\alphar}\del_{\betar},Z^{I^*}]w_{\jr} + Z^{I^*}F_i.
$$

By energy lemma \ref{pre lem energy} and lemma \ref{proof lem curved energy is big}, using the notation of lemma \ref{pre lem energy}, when $|I^*|\leq 4$,
$$
\bigg(\sum_i E_m(s,Z^{I^*}w_i)\bigg)^{1/2} \leq \bigg(\sum_i E_m(B+1,Z^{I^*}w_i)\bigg)^{1/2} + \int_{B+1}^s \sqrt{3}\sum_i L_i(\tau) + \sqrt{3n_0}M(\tau)d\tau.
$$
By lemma \ref{proof lem L2 F 4} and \ref{proof lem L2 com 4},
$$
\sum_iL_i(s) + M(s) \leq C(C_1\varepsilon)^2s^{-1+\delta}
$$
with $C$ a universal constant.
Then
$$
\bigg(\sum_i E_m(s,Z^{I^*}w_i)\bigg)^{1/2} \leq C_{n_0}(C_1\varepsilon)^2\delta^{-1} s^{\delta} + (1/4)C_1\varepsilon.
$$
When $\varepsilon \leq \frac{\delta}{4C_1C_{n_0}}$,
$$
\bigg(\sum_i E_m(s,Z^{I^*}w_i)\bigg)^{1/2}\leq (1/2)C_1\varepsilon s^{\delta},
$$
which leads to
$$
E_m(s,Z^{I^*}w_i) \leq (1/2)C_1\varepsilon s^{\delta}.
$$
Similarly, one derives \ref{main eq main} with respect to $Z^I$ with $|I|\leq 3$:
$$
\Box Z^I u_{\ih} + G_{\ih}^{\jr\alphar\betar}\del_{\alphar\betar}w_{\jr} = [G_{\ih}^{\jr\alphar\betar}\del_{\alphar}\del_{\betar},Z^I]w_{\jr} + Z^IF_{\ih}.
$$

Also by lemma \ref{pre lem energy},
$$
\bigg(\sum_i E_m(s,Z^Iw_i)\bigg)^{1/2} \leq \bigg(\sum_i E_m(B+1,Z^Iw_i)\bigg)^{1/2} + \int_{B+1}^s \sqrt{3}\sum_i L_i(\tau) + \sqrt{3n_0}M(\tau)d\tau.
$$
By lemma \ref{proof lem L2 F 3}, lemma \ref{proof lem L2 com complete} and lemma \ref{proof lem Mv 3}
$$
\sum_iL_i(s) + M(s)\leq C(C_1\varepsilon)^2s^{-3/2+\delta}.
$$
When $\varepsilon \leq \frac{1-4\delta}{4C_1C_{n_0}}$,
$$
E_m(s,Z^Iw_i)^{1/2}\leq (1/2)C_1\varepsilon.
$$
When taking $\varepsilon \leq \min\big\{\frac{\delta}{4C_1C_{n_0}},\,\frac{1-4\delta}{4C_1C_{n_0}}\big\}$
\footnote{Also, taking $C_1\varepsilon<1$ small enough such that lemma \ref{proof lem 00 expression} and lemma \ref{proof lem curved energy is big} hold.}
, non of the equality of \eqref{proof energy assumption last}
holds. This contradiction leads to the desired result.
\end{proof}
\appendix
\section{Local existence for small initial data}
One will establish the following local-in-time existence result for small initial data. The interest is to control the energy on hyperboloid $H_{B+1}$ by the energy on plan $\{t=0\}$.
Consider the Cauchy problem in $\mathbb{R}^{n+1}$:
\begin{equation}\label{appendix eq nonlinear}
\left\{
\aligned
&g_i^{\alpha\beta}(w,\del w)\del_{\alpha\beta} w_i + D_i^2 w_i = F_i(w,\del w),
\\
&w_i(B+1,x) =\varepsilon'{w_i}_0,\quad \del_t w_i(B+1,x) = \varepsilon'{w_i}_1.
\endaligned
\right.
\end{equation}
Here
$$
\aligned
&g_i(w,\del w) = m^{\alpha\beta} + A^{\alpha\beta\gamma j}_i\del_{\gamma}w_j + B^{\alpha\beta j}w_j ,
\\
&F_i(w,\del w) = P^{\alpha\beta jk}_i \del_{\alpha}w_j \del_{\beta}w_k + Q^{\alpha jk}_i \del_{\alpha}w_j w_k + R^{jk}_i w_j w_k .
\endaligned
$$
These $A^{\alpha\beta\gamma j}_i,B^{\alpha\beta j},P^{\alpha\beta jk}_i,Q^{\alpha jk}_i,R^{jk}_i$ are constants.
$({w_i}_0,{w_i}_1)\in H^{s+1}\times H^s$ functions and supported on the disc $\{|x| \leq B\}$. In general the following local-in-time existence holds
\begin{theorem}\label{appendix Thm A}
For any integer $s\geq 2p(n)-1$,
there exists a time interval $[0,T(\varepsilon')]$ on which
the cauchy problem \eqref{appendix eq nonlinear} has an unique solution in sense of distribution $w_i(t,x)$. Further more 
$$
w_i(t,x) \in C([0,T(\varepsilon')], H^{s+1}) \cap C^1([0,T(\varepsilon')], H^s),
$$
and when $\varepsilon'$ sufficiently small,
$$
T(\varepsilon') \geq C(A\varepsilon')^{-1/2}
$$
where $A$ is a constant depending only on ${w_i}_0$ and ${w_i}_1$.
Let $E_g(T,w_i)$ be the hyperbolic energy defined in the section 2.2. For any $\varepsilon, C_1>0$, there exists an $\varepsilon'$ such that
$$
\sum_{i}E_g(B+1, w_i) \leq C_1\varepsilon.
$$
\end{theorem}
\begin{proof}
The proof is just a classical iteration procedure. One will not give the details but the key steps.
One defines the standard energy associated to a curved metric $g$
$$
E^*_g(s,w_i) := \int_{\mathbb{R}^n} \big(g^{00}(\del_t u)^2 - g^{ij}\del_iu\del_ju\big) dx.
$$
One takes the following iteration procedure:
\begin{equation}\label{appendix proof iteration}
\left\{
\aligned
&g^{\alpha\beta}_i(w^k,\del w^k) \del_{\alpha\beta} w_i^{k+1} = F(w^k,\del w^k),
\\
&w^{k+1}_i(0,x) = \varepsilon'{w_i}_0,\quad \del_t w^{k+1}_i(0,x) = \varepsilon'{w_i}_0,
\endaligned
\right.
\end{equation}
and take $w_i^0$ as the solution of the following linear Cauchy problem:
$$
\left\{
\aligned
& \Box w_i = 0,
\\
& w_i(0,x) = \varepsilon'{w_i}_0, \quad \del_t w_i (0,x) = \varepsilon'{w_i}_1.
\endaligned
\right.
$$
Suppose that for any $|I|\leq 2p(n)-1$,
\begin{equation}\label{appendix proof energy assumption}
\aligned
&\varepsilon'A \geq e\cdot E^*_g(B+1,\del^I w_i^k)^{1/2},
\\
&\varepsilon'A \geq E^*_g(t,\del^I w_i^k)^{1/2}.
\endaligned
\end{equation}
Taking the size of the support of the solution $w_i^k(t,\cdot)$ into consideration, by Sobolev's inequality, for any $|J| \leq p(n)-1$,
\begin{equation}\label{appendix proof decay estimate}
|\del^J w_i^k|(t,x) \leq C(t+B+1) \varepsilon' A.
\end{equation}
Now one wants to get the energy estimate on $\del^I w_i^{k+1}$. By the same method used in \cite{So}, one gets
$$
\aligned
E^*_g(t,\del^I w_i^{k+1})^{1/2}
&\leq E_g(t,\del^I w_i^{k+1}) \exp\bigg(CA\varepsilon'\int_{B+1}^t(\tau+B+1)d\tau\bigg)
\\
&\leq e^{-1} \varepsilon'A \exp\bigg(CA\varepsilon'\int_{B+1}^t(\tau+B+1)d\tau\bigg)
\endaligned
$$
When
$$
\sqrt{CA\varepsilon'} \leq (B+1)^{-1}
$$
and
$$
t\leq \frac{1}{3}(CA\varepsilon')^{-1/2},
$$
one gets that
$$
E^*_g(t,\del^I w_i^{k+1})^{1/2} \leq \varepsilon' A.
$$
Then by an standard method presented in the proof of theorem ... of \cite{So},
$$
\lim_{k \rightarrow \infty} w_i^{k} = w_i
$$
is the unique solution of \eqref{appendix eq nonlinear}, and $w_i \in C([0,T(\varepsilon')], H^{s+1}) \cap C^1[0,T(\varepsilon')], H^s)$.
Here one can take
$$
T(\varepsilon') = C(A\varepsilon)^{-1/2}
$$

To estimate $E_g(B+1,Z^I w_i)$, one takes $\del_t w_i$ as the multiplier and by the standard procedure of energy estimate,
$$
\aligned
E_g(B+1,Z^I w_i) - E^*_g(B+1,Z^I w_i)
&= \int_{V(B)} \big(Z^I F_i(w, \del w)\del_t w_i - [Z^I,g^{\alpha\beta}\del_{\alpha\beta}] w_i \cdot \del_t w_i\big) dx
\\
&+ \int_{V(B)}\bigg(\del_{\alpha}g^{\alpha\beta}\del_t w_i\del_{\beta}w_i -\frac{1}{2}\del_t g^{\alpha\beta}\del_{\alpha}w_i \del_{\beta} w_i\bigg) dx,
\endaligned
$$
where $V(B) : = \{(t,x): t\geq B+1, t^2-|x|^2 \leq B+1\}\cap \Lambda'$.
When $A\varepsilon'\leq (B+1)^{-2}$, thanks to \eqref{appendix proof decay estimate} and \eqref{appendix proof energy assumption},
the right hand side can be controlled by $CA\varepsilon'$. Then one gets
$$
E_g(B+1,Z^I w_i)\leq CA\varepsilon'.
$$
\end{proof}
\begin{center}
\bf{\Large{Acknowledgments}}
\end{center}
Part of this work is motivated by a joint work with the author's doctoral supervisor Prof. Ph. LEFLOCH. The author is grateful to him.
The author is also grateful to his parents Dr. Qing-jiu MA and Ms. Huiqin-YUAN, his fianc\'ee Miss Yuan-yi YANG and his comrade Ye-ping ZHANG,
for their successive supports and encouragements.

\end{document}